\documentclass{article}
\usepackage[utf8]{inputenc}

\usepackage{amsmath, amsthm, amsfonts,stmaryrd,amssymb}
\usepackage[top=1in, bottom=1.25in, left=1.25in, right=1.25in]{geometry}
\usepackage{mathrsfs}
\usepackage{dsfont}                  
\usepackage{mathtools}                   
\usepackage{booktabs}
\usepackage{tikz}
\usepackage{dirtytalk}
\usepackage{todonotes}
\usepackage{enumerate} 
\usepackage{hyperref}
\usepackage{theoremref}
\usepackage{verbatim}
\usepackage{bm}
\usepackage{titlefoot}

\usepackage{biblatex}
\usepackage{nomencl}
\let\oldprintnomenclature\printnomenclature
\renewcommand{\printnomenclature}{\begingroup\let\clearpage\relax\let\cleardoublepage\relax\oldprintnomenclature\endgroup}
\makenomenclature

\usepackage{pgfplots}
\pgfplotsset{compat=1.15}
\usepackage{mathrsfs}
\usetikzlibrary{arrows}
\newtheorem{theorem}{Theorem}[section]
\newtheorem*{theorem*}{Theorem}
\newtheorem{remark}[theorem]{Remark}
\newtheorem{lemma}[theorem]{Lemma}

\newtheorem{proposition}[theorem]{Proposition}

\theoremstyle{definition}
\newtheorem{definition}[theorem]{Definition}
\def\ga{\gamma} 
\def\de{\delta} 
\def\ep{\varepsilon}

\def\ka{\kappa} 
 
\def\rh{\rho} 
\def\si{\sigma} 
 
\def\ph{\varphi} 
 
\def\ps{\psi}

\def\Th{\Theta} 
 
\def\Si{\Sigma} 
\def\Ph{\Phi}

\def\P{{\mathcal P}}

\def\H{\mathcal{H}}
\def\R{{\mathbb R}}

\def\Cc{{\mathcal C}}
\def\Ec{{\mathcal E}}

\def\Nc{{\mathcal N}}
\def\N{{\mathbb N}}

\title{Convergence Rates of the Regularized Optimal Transport: \\ Disentangling Suboptimality and Entropy}
\author{Hugo Malamut, Maxime Sylvestre}

\begin{document}

\maketitle

\begin{abstract}
We study the convergence of the \textcolor{black}{entropic optimal} transport plans $\gamma_\ep$ towards the optimal transport plan $\gamma_0$  as well as the cost of the entropy-regularized optimal transport $(c,\gamma_\ep)$ towards $(c,\gamma_0)$ as the regularization parameter $\ep$ vanishes in the setting of finite entropy marginals. We show that under the assumption of infinitesimally twisted cost and compactly supported marginals the distance $W_2(\gamma_\ep,\gamma_0)$ is asymptotically greater than $C\sqrt{\ep}$ and the suboptimality $(c,\gamma_\ep)-(c,\gamma_0)$ is of order $\ep$. In the quadratic cost case the compactness assumption is relaxed into a moment of order $2+\delta$ assumption. Moreover, in the case of \textcolor{black}{existence of} a Lipschitz transport map for the non-regularized problem, the distance $W_2(\gamma_\ep,\gamma_0)$ converges to $0$ at rate $\sqrt{\ep}$. Finally, if in addition the marginals have finite Fisher information, we prove $(c,\gamma_\ep)-(c,\gamma_0) \sim d\ep/2$ and we provide a companion expansion of $H(\gamma_\ep)$. These results are achieved by disentangling the role of the cost and the entropy in the regularized problem.
\end{abstract}



\tableofcontents

\nomenclature[01]{$\Omega,X$}{$\Omega$ is an open subset of $\mathbb{R}^d$ and $X$ a compact subset of $\Omega$}
\nomenclature[02a]{$\P(\Omega)$}{Probability measures on $\Omega$}
\nomenclature[02b]{$\P_{ac}(\Omega)$}{\textcolor{black}{Probability measures on $\Omega$ which are absolutely continuous with respect to the Lebesgue measure}}
\nomenclature[02c]{$\P_2(\Omega)$}{\textcolor{black}{Probability measures on $\Omega$ with finite variance}}
\nomenclature[03]{$\H^k$}{Hausdorff measure of dimension $k$}
\nomenclature[04]{$W_2(\mu,\nu)$}{2-Wasserstein distance between $\mu$ and $\nu$}
\nomenclature[05]{$H(\mu\mid\nu)$}{relative entropy of $\mu$ with respect to $\nu$ (see equation \eqref{eq:entropy&I}). Sometimes $H(\mu) : = H(\mu\mid \H^k)$ for some $k \in \N$}
\nomenclature[06]{$I(\mu)$}{Fisher information of $\mu$, see equation \eqref{eq:entropy&I}}

\nomenclature[07]{$\Vert \Ph \Vert_{op}$}{Operator norm of the affine transformation $\Ph$}
\nomenclature[08]{$\Pi(\mu_0,\mu_1)$}{Transport plans between $\mu_0$ and $\mu_1$}
\nomenclature[09]{$c$}{Cost function, $c\in \Cc^2(\Omega)$. In section 2 and 3, $c(x,y) :=\frac{1}{2}\Vert x-y\Vert^2$ (quadratic cost)}

\nomenclature[11]{$T = \nabla f$}{Brenier's optimal map from $\mu_0$ to $\mu_1$, $f$ is a convex function}
\nomenclature[12]{$\ph,\ps$}{a pair of Kantorovich potentials (see equation \eqref{eq:duality-kantorovich})}
\nomenclature[13]{$E$}{a duality gap function defined by $E: = c - (\ph \oplus \ps)$}
\nomenclature[14]{$\Si$}{a set of optimal pairings $\Si := \{(x,y) / E(x,y) = 0\}$}
\nomenclature[16]{$OT_\ep$}{$\ep$-entropic transport cost from $\mu_0$ to $\mu_1$ (see equation \eqref{eq:EOT})}
\nomenclature[17]{$\gamma_\ep$}{Optimal transport plan for $\ep$-Entropic Optimal Transport ($\ep$EOT)}
\nomenclature[18]{$v^\ep_t,\rh_t^\ep$}{Solutions of $\ep$-entropic Benamou Brenier formula ($\ep$BB)}

\nomenclature[24]{$N_k(\mu)$}{Entropy power function of $\mu$ : if  $H(\mu \mid \H^k) < + \infty$, then $N_k(\mu) = \frac{1}{2\pi e}e^{-\frac{2}{k}H(\mu \mid \H^k)}$}
\nomenclature[28]{$m_k(\mu)$}{ Moment of order $k$ of the probability $\mu$}

\nomenclature[32]{$H_m$}{ Mean entropy $H_m := \frac{1}{2} (H(\mu_0\mid \H^d) + H(\mu_1\mid \H^d))$}
\nomenclature[33]{$C(a,b,c)$}{ A constant that only depends on the terms $a$,$b$ and $c$}

\unmarkedfntext{We would like to thank G. Carlier and P. Pegon for fruitful discussions and helpful advices.}

\printnomenclature

\section{Introduction}
We study the regularized optimal transport problem for a cost $c \in \mathcal{C}^2(\mathbb{R}^d\times \mathbb{R}^d, \mathbb{R})$ 
\begin{equation}
\inf\limits_{\ga \in \Pi(\mu_0,\mu_1)} \int cd\ga + \ep H(\ga \vert \mu_0 \otimes \mu_1),
\end{equation}
where the infimum is taken over all measures $\ga \in \mathcal{P}(\mathbb{R}^d \times \mathbb{R}^d)$ with marginals $\mu_0,\mu_1$. Here $H(.|.)$ is the relative entropy also known as Kullback-Leibler divergence. In this paper we will focus on the case where $\mu_0,\mu_1$ have finite entropy with respect to $\H^d$, the Lebesgue measure, $H(\mu_i\mid \mathcal{H}^d) < \infty$ and finite order two moments. In that case the minimizer $\ga_\ep$ is the same as if $\mu_0\otimes\mu_1$ was replaced by the Lebesgue measure $\mathcal{H}^{2d}$ in the entropy term, and hence we consider
\begin{equation}\label{eq:EOT}
    OT_\ep := \inf_{\ga \in \Pi(\mu_0,\mu_1)} \int c d\ga + \ep H(\ga \vert \mathcal{H}^{2d}). \tag{$\ep$EOT}
\end{equation}
Note that $\ep = 0$ yields the classical optimal transport problem. In this case the minimizer need\textcolor{black}{s} not to be unique and $\ga_0$ will sometimes denote any of them. We are interested in deriving rates for the cost term $\int c d \ga_\ep = (c,\ga_\ep)$ as well as the 2-Wasserstein distance between  $\ga_\ep$ and $\ga_0$, when \textcolor{black}{the latter} is unique.\\

In the last decade this problem has witnessed a rapid increase in interest. It has proved to be an efficient way to approximate OT problems, especially from a computational viewpoint. The celebrated Sinkhorn's algorithm \cite{sinkhorn1964sink} was applied in this framework in the pioneering works  \cite{cuturi2013compsink,ben+2015sink}. The good convergence guarantees \cite{franklin1989presinkh,marino2020disc} cemented the success of EOT and its applications.\\
Clearly EOT is a perturbation of classical OT thus it is natural to study the behaviour of this problem as $\ep$ vanishes. In this direction several aspects deserved to be studied such as the convergence of optimal values, potentials (optimizers of the dual problem) and optimal plans, possibly with quantitative rates. 
In the direction of convergence of optimal values recent contributions have thoroughly treated the issue: in the quadratic case, under regularity assumptions, the link between EOT and the Schrödinger problem \cite{leonard} has allowed to find a second order \cite{erbar2015large} and more recently a third order \cite{conforti2019formula,chizat2020faster} expansion of the value $OT_\ep$ in $\ep$. The second order expansion has been generalized to other cost functions \cite{pal2019difference}, and the first order term has been obtained under very mild assumptions on the cost function and the marginals \cite{carlier2022convergence,nutz2022rate}. Those articles focus on the value \textcolor{black}{of the} problem \eqref{eq:EOT}.\\

Our main objective is to disentangle the role of the cost $\int c d\gamma_\ep$ and the entropy \textcolor{black}{$H(\gamma_\ep\mid \mathcal{H}^{2d})$} in order to derive rate of convegence for both. The cost term is of interest itself because it is a faster converging approximation of $OT_0$. The entropy term also allows to lower bound $W_2(\ga_\ep,\ga_0)$, the Wasserstein distance between the entropic optimal transport plan $\gamma_\ep$ and the optimal transport plan $\gamma_0$. The study of the convergence of the cost term in EOT has also been done recently in \cite{altschuler2022semidisc} for the semidiscrete case which grants an $\ep^2$ rate. \textcolor{black}{The authors} also derive convergence rates for the Kantorovich potentials. In the discrete case the rate of convergence is exponential \cite{cominetti1994dicreterate,weed2018discreterate}. In the continuous case, to our knowledge, no asymptotic rate for the suboptimality $\int c d\gamma_\ep - \int c d \ga_0$ was known but a rate of order $\ep$ was suspected based on simple examples such as Gaussian measures. We tackle the problem of sizing the suboptimality in this article. \\

\subsection{Main results}
The aim of this paper is to provide tight rates of convergence on the Wasserstein distance $W_2(\ga_\ep,\ga_0)$, the cost term $(c,\ga_\ep)$  and the entropy term $H(\gamma_\ep\mid \mathcal{H}^{2d})$.
Let $\mu_0,\mu_1 \in \mathcal{P}_{ac}(\mathbb{R}^d)$ be such that $H(\mu_i\mid \mathcal{H}^d) < \infty$. It is known \cite{carlier2022convergence,nutz2022rate} that under mild assumptions (see Lemma \ref{lem:otep-ot-constant-in-rate}) on $\mu_0,\mu_1$,
\begin{equation}
    OT_\ep = (c,\ga_\ep) + \ep H(\ga_\ep\mid\mathcal{H}^{2d}) \leq OT_0 - \frac{d}{2}\ep\ln(\ep) + O(\ep).
\end{equation}
For the quadratic cost and under the stronger assumption of finite Fisher information for the marginals, we have (see \cite[Claim 4.1]{erbar2015large}, see also \cite{pal2019difference} for other costs and different hypothesis)
\begin{equation}
    OT_\ep = OT_0 - \frac{d}{2} \ep \textcolor{black}{\ln}(2\pi\ep) + \ep H_m + o(\ep).
\end{equation}
where $H_m := \frac{1}{2}[H(\mu_0\mid \mathcal{H}^{d})+ H(\mu_1\mid \mathcal{H}^{d})]$.
Our goal \textcolor{black}{is} to disentangle the role of the cost term $(c,\ga_\ep)$ and of the entropy term $H(\gamma_\ep\mid \mathcal{H}^{2d})$ in those \textcolor{black}{rates} of convergence. \\
For clarity we present here the three main results present in the article. The first gives a second order expansion of the entropy and the cost of the entropic optimal transport. It relies on known expansions of the Benamou-Brenier \eqref{eq:BB} formulation of the entropic optimal transport problem. The proof of this result can be found in Section \ref{sec:exact asymptotics}
\begin{theorem*}[Theorem \ref{thm:expansions}]
    Suppose that the cost is quadratic, that is $c(x,y) = \frac{1}{2}\Vert x-y \Vert^2$. Further assume that $I(\mu_i)< \infty$ and $\text{supp}(\mu_i)$ compact. Then
    \begin{equation}
    H(\gamma_\ep\mid \mathcal{H}^{2d}) = - \frac{d}{2} \textcolor{black}{\ln}(2\pi\ep) + H_m - \frac{d}{2} + o(1)
    \end{equation}
    and
    \begin{equation}
     (c,\ga_\ep) = OT_0 + \frac{d}{2} \ep + o(\ep),
 \end{equation} 
\end{theorem*}
If we relax the finite Fisher information hypothesis the orders of magnitude still hold. The method, presented in Section \ref{sec:quadratic-cost-rates}, used to prove this result relies on the Minty reparametrization trick (see definition \ref{def:Minty}) which implies a quadratic detachment for the duality gap function $E = c - (\ph \oplus \ps)$, where $\ph,\ps$ are Kantorovich potentials.
\begin{theorem*}[Theorem \ref{thm:rate-cost-quadratic}]
    Suppose that the cost is quadratic, that is $c(x,y) = \frac{1}{2}\Vert x-y \Vert^2$. Further assume that $\mu_i$ have finite moment of order $2+\delta$ then
    \begin{equation}
    (c,\gamma_\ep) = OT_0 + \Th(\ep), \quad  H(\gamma_\ep\mid \mathcal{H}^{2d}) =- \frac{d}{2} \ln(\ep) + O(1), \quad \sqrt{\ep} = O(W_2(\gamma_\ep,\gamma_0)).
    \end{equation}
    In the special case where the Monge map $\nabla f$ associated to the optimal transport plan $\ga_0$ is Lipschitz then
    \begin{equation}
        W_2(\gamma_\ep,\gamma_0) = \Th(\sqrt{\ep}).
    \end{equation}
\end{theorem*}
Finally for infinitesimally twisted costs we have a similar result but under the stronger assumption of compactly supported marginals. The concept of local quadratic detachment \textcolor{black}{(Definition \ref{def:loc-quadratic-detachment})} is introduced and is key to the proof of the result which can be found in Section \ref{sec:general-rates}.
\begin{theorem*}[Theorem \ref{thm:general-rates}, Theorem \ref{thm:gen-mino-wass-any-coupling}]
    Suppose that the cost is $\mathcal{C}^2$ and infinitesimally twisted (see definition \ref{def:twist}). Further assume that $\mu_i$ is compactly supported then
    \begin{equation}
    (c,\gamma_\ep) = OT_0 + \Th(\ep), \quad  H(\gamma_\ep\mid \mathcal{H}^{2d}) =- \frac{d}{2} \ln(\ep) + O(1), \quad \sqrt{\ep} = O(W_2(\gamma_\ep,\gamma_0)).
    \end{equation}
\end{theorem*}
\textcolor{black}{Note that in the case of an infinitesimally twisted cost the solution to the optimal transport problem is not necessarily unique. Thus we stress that the lower bound on the Wasserstein distance holds for any solution $\ga_0$ of the optimal transport problem.}

\subsection{Definitions and assumptions}

First, let \textcolor{black}{us} recall the definition of the \emph{relative entropy} $H(\mu|\nu)$ and the \emph{Fisher information} $I(\mu)$ that provide quantitative estimate of the smoothness of a measure: if $\mu,\nu \in \P_{ac}(\R^d)$, we note
\begin{equation}\label{eq:entropy&I}
    H(\mu | \nu) := \left\{ \begin{array}{c}
   \int \ln\left(\frac{d\mu}{d\nu}\right) d\mu \emph{   if $\mu \ll \nu$} \\
    + \infty \emph{   otherwise}
\end{array}  \right. \quad \text{and} \quad  
I(\mu) = \left\{ \begin{array}{c}
   \int \frac{\| \nabla \mu(x) \|^2}{\mu(x)^2} d\mu(x)  \emph{   if $\mu \ll \H^d$} \\
    + \infty \emph{   otherwise}
\end{array}  \right. 
\end{equation}
\textcolor{black}{where} $\frac{d\mu}{d\nu}$ denotes the Radon-Nikodym derivative of $\mu$ with respect to $\nu$ and $\mu(x)$ the density of $\mu$ with respect to the Lebesgue measure $\H^d$ by a slight abuse of notation. \textcolor{black}{Following \cite[Annexe A]{leonard}, we will use the notation $H(\mu|\nu)$ even when the measure $\nu$ is not necessarily a probability measure. Typically, when $\mu\in \P_{2,ac}(\R^d)$, we will use the notation $H(\mu | \H^d) := \int \ln(\mu(x)) d\mu(x)$ for its differential entropy. Indeed, as explained in \cite[Annexe A]{leonard}, as soon as the second moment of $\mu$ is finite, its differential entropy $H(\mu | \H^d)$ is well defined. } Note that a finite Fisher information $I(\mu) < + \infty$ implies a finite differential entropy $H(\mu | \H^d)$ \cite[Chapter 9]{villani2003topics}, that we will often denote simply by $H(\mu)$ .\\
We now introduce the \emph{Kantorovich potentials} which are the solutions to the dual of the classical optimal transport problem
\begin{equation}\label{eq:duality-kantorovich}
    \sup_{\ph,\ps} \int\ph d\mu_0 + \int \ps d \mu_1,
\end{equation}
\textcolor{black}{where} the $\sup$ is taken over all functions $\ph \in L^1(\mu_0), \ps \in L^1(\mu_1)$ such that for all $(x,y) \in \R^d \times \R^d$, we have $\ph(x) + \ps(y) \leq c(x,y)$. It is well known that in our current setting this problem has maximizers which are continuous \cite[Chapter 1]{villani2003topics}, that is the Kantorovich potentials are continuous. Moreover it is also possible to assume that the constraint $\ph(x) + \ps(y) \leq c(x,y)$ is binding, that is that $\ph$ and $\ps$ are \emph{c-transform}\footnote{$\ph(x) = \inf_y c(x,y) - \ps(y)$ and $\ps(y) = \inf_x c(x,y) - \ph(x)$} one of each other. In the following, we will use the name \emph{Kantorovich potentials} to denote a pair of continuous and mutual c-transform solutions to the dual problem \eqref{eq:duality-kantorovich}.\\
For a pair $(\ph,\ps)$ of Kantorovich potentials, we will call \emph{duality gap} the quantity $E(x,y):= c(x,y) - (\ph(x) + \ps(y))$ and we will note \textcolor{black}{$E = c - (\ph \oplus \ps)$}. Remark that in the quadratic case, when $\mu_i \in \P_{ac}(\R^d)$, the Kantorovich potentials are unique up to a constant and so the duality gap function E is unique. \\
In the general case, we will be interested in the gap between $\int c d\ga_\ep$ and $\int c d\ga_0$ which can be restated in terms of any duality gap function $E = c - (\ph \oplus \ps)$ with $\ph,\ps$ Kantorovich potentials. Indeed we have by duality that $\int c d\ga_0 = \int \ph d\mu_0 + \int \ps d\mu_1$ and so 
\begin{equation}
    \int c d\ga_\ep - \int c d\ga_0 = \int c d\ga_\ep - \int (\ph \oplus \ps) d\ga_\ep = \int E d\ga_\ep.
\end{equation}
As explained in the main results, we will explore three different sets of assumptions. Each set of \textcolor{black}{assumptions} corresponds to one of the Theorems presented in the main results.
\begin{itemize}
 
 \item {\bf{(H1)}} The cost is quadratic: $c(x,y) := \frac{1}{2}\|x-y\|^2$. \\
 The \textcolor{black}{marginals} have finite \textcolor{black}{Fisher} information $I(\mu_i) < + \infty$ and compact support
 
 \item  {\bf{(H2)}} The cost is quadratic: $c(x,y) := \frac{1}{2}\|x-y\|^2$. \\
 The marginals have finite differential entropy $H(\mu_i | \H^d) < + \infty$ and finite moments of order $2+\de$ for some $\de > 0$ 
 
 \item {\bf{(H3)}} The cost $c$ belongs to $\Cc^2$ and is infinitesimally twisted (see definition \ref{def:twist}). \\ 
 The marginals $(\mu_i)$ have compact support and finite differential entropy.
\end{itemize}

Since a finite Fisher information implies a finite differential entropy, hypothesis $\bf{(H2)}$ is an important relaxation of $\bf{(H1)}$ on the regularity and concentration of $(\mu_i)$. Similarly, hypothesis $\bf{(H3)}$ is an important relaxation of $\bf{(H1)}$ since the class of cost is broader and the marginals are assumed less regular.
A last remark : while the finite entropy or finite information hypothesis are deeply linked with the nature of the problem and essential to our results, the concentration assumptions such as compact support or finite moment must be seen as technical and the results are likely to hold in any case.

\section{Exact asymptotics from the Schrödinger problem}\label{sec:exact asymptotics}

In this Section, we assume validity of hypothesis $\bf{(H1)}$, that is that the cost is quadratic and that the marginals $\mu_0$ and $\mu_1$ have finite Fisher information and compact support. Under finite Fisher information of the marginals, it is known that the value of \eqref{eq:EOT} has a second order expansion in $\ep$ (see \cite[Claim 4.1]{erbar2015large}, see also \cite{pal2019difference} for other costs and different hypothesis)

\begin{equation}\label{eq:expanOTep}
    OT_\ep = \frac{1}{2} \textcolor{black}{W_2^2(\mu_0,\mu_1)} - \frac{d}{2} \ep \textcolor{black}{\ln}(2\pi\ep) + \ep H_m + o(\ep) \quad (\ep \rightarrow 0).
\end{equation}
In this Section, we will disentangle the roles of $(c,\ga_\ep)$ and $H(\ga_\ep | \mathcal{H}^{2d})$ in this asymptotic formula, and hence give a Taylor expansion of both. We will use the dynamic formulation of the entropic optimal transport.

\subsection{ Schrödinger problem and Benamou-Brenier formulation}\label{sec:schrodinger}

The quadratic EOT problem can be reformulated as a \emph{Schrödinger problem} (see \cite{leonard})

\begin{equation}\label{eq:schrodinger}
    \Cc_\ep = \min_{\ga \in \Pi(\mu_0,\mu_1)} H(\ga | m_\ep)
\end{equation}
where $m_\ep$ \textcolor{black}{is} the measure of density $\frac{1}{(2\pi \ep)^{d/2}}e^{-c/\ep}$ with respect to the Lebesgue measure $\H^{2d}$. Since for $\ga \in \Pi(\mu_0,\mu_1)$, we can compute \textcolor{black}{explicitly} $\ep H(\ga | m_\ep) = (c,\ga) + \ep H(\ga\mid\mathcal{H}^{2d}) + \frac{d}{2} \ep \ln(2\pi\ep)$\footnote{\textcolor{black}{Note that this can be considered as a definition for $H(\ga | m_\ep)$, since $m_\ep$ is not a probability measure}}, we have that the solutions of \eqref{eq:schrodinger} and \eqref{eq:EOT} are the same and
\begin{equation}\label{eq:EOT-schrodinger}
    OT_\ep = \ep \Cc_\ep - \frac{d}{2} \ep \ln(2\pi\ep).
\end{equation}
Moreover there exists an analogous to Benamou-Brenier formula for $\Cc_\ep$ (see \cite{gigli2020benamou}):

\begin{equation}\label{eq:BB}
    \Cc_\ep = H_m + \min_{\rh,v} \frac{1}{\ep} \int_0^1 \int \frac{1}{2} |v_t|^2 d\rh_t(x) dt + \frac{\ep}{8} \int_0^1 \int \frac{\|\nabla \rh_t\|^2}{\rh_t} dx dt \tag{$\ep$BB}
\end{equation}
\textcolor{black}{where} the $\min$ is taken as in the classical Benamou-Brenier formula on paths from $\mu_0$ to $\mu_1$ that in the weak sense solve the continuity equation $\partial_t \rh + div(\rh v) = 0$. We will denote by $v^\ep, \rh^\ep$ the solution. This dynamic formulation and other variations are in fact the core of the Schrödinger problem, that intend to solve for the most probable trajectory of a free process at positive temperature with fixed initial and final marginals $\mu_0$ and $\mu_1$ (see \cite{leonard} for a survey). The first term in \eqref{eq:BB} corresponds to the kinetic energy while the second term corresponds to the diffusion ($\ep$ is the temperature parameter).

\subsection{A precise expansion of $H(\ga_\ep)$ and $(c,\ga_\ep)$}\label{sec:precise expansion}

Minimization problems \eqref{eq:EOT} and \eqref{eq:BB} give the natural idea of using the envelope Theorem to express the $\ep$-derivative of $OT_\ep$ and $\Cc_\ep$. For a minimization problem with a real parameter $\ep$, this Theorem ensures that at any differentiability point $\ep_0$ of the value function and for any minimizer, the derivative of the value function \textcolor{black}{coincides} with the $\ep$-partial derivative of the objective function \textcolor{black}{evaluated at the $\ep$-minimizer} (see \cite{milgrom2002envelope}). This idea has been followed in \cite{carlier2022convergence} for \eqref{eq:EOT} and in \cite{conforti2019formula} for \eqref{eq:BB}. We will use the more general setting of \cite{carlier2022convergence} that only requires compact support for the marginals $\mu_i$ and obtain that $\ep \mapsto OT_\ep$ is $C^\infty$ for $\ep > 0$. From equation \eqref{eq:EOT-schrodinger}, it is direct that in that case the function $\ep\Cc_\ep$ is also smooth, and so we can \textcolor{black}{apply the envelope theorem} both to $OT_\ep = \min_{\ga} \int \frac{1}{2} \Vert x-y \Vert^2d\ga + \ep H(\ga \mid \mathcal{H}^{2d})$ and to  $\ep \Cc_\ep = \ep H_m + \min_{\rh,v} \int_0^1 \int \frac{1}{2} |v_t|^2 d\rh_t(x) dt + \frac{\ep^2}{8} \int_0^1 I(\rh_t) dt$ to get

\begin{equation}\label{eq:ep_derivative}
    \frac{d}{d\ep} OT_\ep = H(\ga_\ep) \quad and \quad \frac{d}{d\ep} [ \ep \Cc_\ep ]= H_m + \frac{\ep}{4} \int_0^1 I(\rh^\ep_t) dt
\end{equation}
\textcolor{black}{where} $H(\ga_\ep)$ stands for $H(\ga_\ep \mid \mathcal{H}^{2d})$.
But from equation \eqref{eq:EOT-schrodinger} $\frac{d}{d\ep} OT_\ep = \frac{d}{d\ep} \left[\ep \Cc_\ep - \frac{d}{2} \ep \ln(2\pi\ep) \right]$, so we have

\begin{equation}\label{eq:HfromI}
    H(\ga_\ep) = -\frac{d}{2}\ln(2\pi\ep) + H_m - \frac{d}{2} + \frac{\ep}{4} \int_0^1 I(\rh^\ep_t) dt.
\end{equation}
This gives a quantitative relation between the entropy $H(\ga_\ep)$ of the static problem \eqref{eq:EOT} and the average \textcolor{black}{Fisher information} $\int_0^1 I(\rh^\ep_t) dt$ of the dynamic problem \eqref{eq:BB}. It allows to state a complete \textcolor{black}{coupled system} between the different terms of \eqref{eq:EOT} and \eqref{eq:BB} as explained in the following Proposition:
\begin{proposition}\label{prop:fisher_expansion1}
For the quadratic cost, suppose that the support of $\mu_i$ are compact, and that $H(\mu_i | \H^d) < + \infty$. Then
\begin{equation}\label{eq:system}
\left\{ \begin{array}{ccc}
   (c,\ga_\ep) &=& \displaystyle \int_0^1  \int \frac{1}{2} |v_t^\ep|^2 d\rh_t(x) dt - \frac{\ep^2}{8} \textcolor{black}{\int_0^1 I(\rh^\ep_t) dt} + \frac{d}{2}\ep \\
    H(\ga_\ep) &=& \displaystyle \frac{\ep}{4} \int_0^1 I(\rh^\ep) dt  -\frac{d}{2}\ln(2\pi\ep) + H_m -\frac{d}{2} 
\end{array}  \right. .
\end{equation}
\end{proposition}
\begin{proof}
The second line of the system \eqref{eq:system} is nothing but equation \eqref{eq:HfromI}. For the first line, equation \eqref{eq:EOT-schrodinger} implies that 
\begin{equation}
   (c,\ga_\ep) = \ep \Cc_\ep - \frac{d}{2}\ep \ln(2\pi\ep) - \ep H(\ga_\ep).
\end{equation}
We can inject the expression of $H(\ga_\ep)$ given by \eqref{eq:HfromI} and the expression of $\Cc_\ep$ with the minimal continuous path $(\rh^\ep,v^\ep)$ of equation \eqref{eq:BB} to obtain the first line of the system \eqref{eq:system}.
\end{proof}
Thanks to the system \eqref{eq:system}, it is enough to study the asymptotics of \textcolor{black}{$\ep \int_0^1 I(\rh^\ep_t) dt$} and $\int_0^1 \int |v_t^\ep|^2 d\rh_t(x) dt$ to get results on $(c,\ga_\ep)$ and $H(\ga_\ep)$. When the Fisher information of the marginal is finite, it turns out that the behaviour of \textcolor{black}{$\int_0^1 I(\rh^\ep_t) dt$} and $\int_0^1 \int |v_t^\ep|^2 d\rh_t(x) dt$ can be described quite precisely:
\begin{proposition}\label{prop:fisher_expansion2}
Suppose that the cost is quadratic. If $\mu_0$ and $\mu_1$ have finite Fisher information \textcolor{black}{and compact supports, then} when $\ep$ tends to 0 :
\begin{equation}\label{eq:convergence_fisher}
\ep \textcolor{black}{\int_0^1 I(\rh^\ep_t) dt} \rightarrow 0 \quad and \quad 
\int_0^1 \int |v_t^\ep|^2 d\rh_t(x) dt = \textcolor{black}{W_2^2(\mu_0,\mu_1)} + o(\ep).
\end{equation}
\end{proposition}

\begin{proof}
Since the Fisher information of the \textcolor{black}{marginals} is finite, from \cite[Claim 4.1]{erbar2015large} the Taylor expansion \eqref{eq:expanOTep} of $OT_\ep$ holds. But since $OT_\ep = \ep \Cc_\ep - \frac{d}{2} \ep \ln(2\pi\ep)$, expansion \eqref{eq:expanOTep} is equivalent to the following:
\[ \Cc_\ep - \frac{1}{2\ep} \textcolor{black}{W_2^2(\mu_0,\mu_1)} - H_m \rightarrow 0.
\]
However from the expression of $\Cc_\ep$ with the minimal continuous path $(\rh^\ep,v^\ep)$ of equation \eqref{eq:BB},
\[ \Cc_\ep - \frac{1}{2\ep} \textcolor{black}{W_2^2(\mu_0,\mu_1)} - H_m = \frac{1}{\ep} \left( \int_0^1 \int \frac{1}{2} |v_t^\ep|^2 d\rh_t(x) dt - \frac{1}{2} \textcolor{black}{W_2^2(\mu_0,\mu_1)} \right) + \frac{\ep}{8} \int_0^1 \int \frac{\|\nabla \rh_t^\ep\|^2}{\rh_t^\ep} dx dt.
\]
Both terms $\frac{1}{\ep} \left[ \int_0^1 \int \frac{1}{2} |v_t^\ep|^2 d\rh_t(x) dt - \frac{1}{2} W_2^2(\mu_0,\mu_1) \right]$ and $\frac{\ep}{8} \int_0^1 \int \frac{\|\nabla \rh_t^\ep\|^2}{\rh_t^\ep} dx dt$ are positive, the first one because $(\rh^\ep,v^\ep)$ is a path from $\mu_0$ to $\mu_1$ solving the continuity equation, and so the ($\ep =0$) Benamou-Brenier formula holds.
Since the sum tends to zero, we know that both terms tend to zero which is what we needed to prove.
\end{proof}

Proposition \ref{prop:fisher_expansion2} can be seen as a refinement of \cite[Lemma 3.3]{gentil2020entropic} where, among other asymptotics, it is shown that $\ep^2 \textcolor{black}{\int_0^1 I(\rh^\ep_t) dt} \rightarrow 0$ and $\int_0^1 \int |v_t^\ep|^2 d\rh_t(x) dt \to \textcolor{black}{W_2^2(\mu_0,\mu_1)} $.
We can now combine Propositions \ref{prop:fisher_expansion1} and \ref{prop:fisher_expansion2} to obtain a Taylor expansion of $(c,\ga_\ep)$ and $H(\ga_\ep)$:

\begin{theorem}\label{thm:expansions}
Suppose that the cost is quadratic, let $\mu_0,\mu_1 \in \P_{ac}(\R^d)$ and suppose that for $i =1,2$, we have  $I(\mu_i) < + \infty$ and the support of $\mu_i$ is compact. Then
\begin{equation}\label{eq:H(ga_ep)}
    H(\ga_\ep) = - \frac{d}{2} \textcolor{black}{\ln}(2\pi\ep) + H_m - \frac{d}{2} + o(1)
\end{equation}
and
 \begin{equation}\label{eq:(c,ga_ep)}
     (c,\ga_\ep) = \frac{1}{2} W_2^2 + \frac{d}{2} \ep + o(\ep).
 \end{equation} 
\end{theorem}

\begin{proof}
    Since a finite Fisher information implies a finite differential entropy, the hypothesis of Propositions \ref{prop:fisher_expansion1} and \ref{prop:fisher_expansion2} hold and we can inject the asymptotics \textcolor{black}{\eqref{eq:convergence_fisher}} on \textcolor{black}{$\int_0^1 I(\rh^\ep_t) dt$} and $\int_0^1 \int |v_t^\ep|^2 d\rh_t(x) dt$  obtained in the latter in the system \eqref{eq:system} of the former.
\end{proof}
\begin{remark}$\quad$

  The rate in $\frac{d}{2} \ep$ for the \textcolor{black}{suboptimality} $(c,\ga_\ep)-(c,\ga_0)$ had been already observed in the case of \textcolor{black}{Gaussian} measures in 1D (see introductions of \cite{altschuler2022semidisc},\cite{bernton2022entropic}). It is astonishing to remark that this \textcolor{black}{first-order} term does not depend on $\mu_0$ and $\mu_1$. Hence, the estimator $(c,\ga_\ep) - \frac{d}{2} \ep$ could provide an interesting estimate of $W_2^2(\mu_0,\mu_1)$ with a lack of precision comparable to the one of the Sinkhorn \textcolor{black}{divergence} $OT_\ep(\mu_0,\mu_1) -\frac{1}{2}[OT_\ep(\mu_0,\mu_0) + OT_\ep(\mu_1,\mu_1)]$ (see \cite{Feydy}).
\end{remark}

\textcolor{black}{In Theorem \ref{thm:expansions}, the compact support assumption is somewhat technical. As mentioned above, the Gaussian measures provide a class of examples for which the solution $\ga_\ep$ is explicit. In this case, the result holds while the compact support assumption is violated. The sharpness of the other assumptions (quadratic cost and finite Fisher information) is harder to track.}

\subsection{Rephrasing the results with different functionals}
The asymptotic formula for the entropy \eqref{eq:H(ga_ep)} and the formula for $(c,\ga_\ep)$ in the system \eqref{eq:system} suggest to introduce two quantities that are linked to our problem. First, the \emph{Entropy power function} $N_d(\mu) := \frac{e^{-\frac{2}{d} H(\mu\mid\mathcal{H}^d)}}{2\pi e}$. This functional has various links with our problem: Costa's Theorem states that it is concave along the heat flow (see \cite{costa1985new} \textcolor{black}{and \cite{newcosta2022tamanini}  for a generalization}), $\sqrt{N_{d}}$ is concave along Wasserstein geodesics (see \cite{erbar2015equivalence}) and along Schrödinger bridges (see \cite{ripani2019convexity}). And it turns out that the \textcolor{black}{asymptotics} \eqref{eq:H(ga_ep)} can be rewritten in the simple way\footnote{Remark that since the functional $N$ is multiplicative, the \textcolor{black}{right-hand side} is the $4d$ entropy power of $\mu_0 \otimes \mu_1 \otimes \Nc(0,\frac{\ep}{2d}I_{2d})$.}
\[
N_{2d}(\ga_\ep) \sim N_d(\mu_0)^{1/4}N_d(\mu_1)^{1/4}\ep^{1/2}.
\]

The other quantity that we need to present is the \emph{energy}\footnote{See the introduction of \cite{conforti2019formula} for a discussion on the name energy.}
\[ \Ec_\ep := \displaystyle \frac{1}{\ep^2} \int \frac{1}{2} |v_t^\ep|^2 d\rh_t(x) - \frac{1}{8} I(\rh_t^\ep).
\]
As the notation hints and as it is proved in \cite[Corollary 1.1]{conforti2019second} and \cite[Lemma 3.3]{gentil2020entropic}, this energy does not depend on the time t.
With this quantity, the expression of $(c,\ga_\ep)$ in the system \eqref{eq:system} becomes $(c,\ga_\ep) = \ep^2 \Ec_\ep + \frac{d}{2}\ep$, and so the Taylor expansion \eqref{eq:(c,ga_ep)} expresses
\[
\ep^2 \Ec_\ep = \frac{1}{2} W_2^2 + o(\ep).
\]

Note that it was already known that $\ep^2 \Ec_\ep \to \frac{1}{2}W_2^2$ (see the introduction of \cite{conforti2019formula} or the proof of \cite[Lemma 3.3]{gentil2020entropic}). The convergence in $o(\ep)$ seems new to us. Given the expression for $\frac{d}{d\ep}[\ep \Cc_\ep]$ of equation \eqref{eq:ep_derivative}, the other Taylor expansion \eqref{eq:H(ga_ep)} corresponding to \textcolor{black}{$\ep \int_0^1 I(\rh^\ep_t) dt \to 0$} is also equivalent to the fact that $\ep\Cc_\ep$ is $\Cc^1$ in 0 with derivative $H_m$. It can be seen as an extension in 0 of \cite[Theorem 1.1]{conforti2019second}.
\section{Rates for the quadratic cost}\label{sec:quadratic-cost-rates}
In the \textcolor{black}{previous} section, under the strong assumptions $\bm{(H1)}$ of smoothness ($I(\mu_i) < + \infty$) and concentration ($\text{supp}(\mu_i)$ compact), the Taylor expansions of \eqref{eq:EOT} and \eqref{eq:BB} allowed us to give precise rates for the entropy as well as the cost term in \eqref{eq:EOT}. The main idea was to disentangle the role of the entropy and the cost. In this section we choose the weaker set of assumptions $\bm{(H2)}$. The cost is still the quadratic one: $c(x,y) = \frac{1}{2}\Vert x-y \Vert^2$, but the Fisher information \textcolor{black}{needs} not to be finite and so the \textcolor{black}{second-order Taylor expansion} \eqref{eq:expanOTep} of \eqref{eq:EOT} needs not to hold. However, the entropy $H(\mu_i)$ is still supposed finite. Thanks to the technical assumption that the moments of order $2+\de$ of the marginals are finite for some $\de>0$, we still have (see Lemma \ref{lem:otep-ot-constant-in-rate}) some rates of convergence for the value of \eqref{eq:EOT} of the form 
\begin{equation} \label{eq:premier_ordre_EOT}
    0 \leq OT_\ep - OT_0 \leq -\frac{d}{2}\ep \ln(\ep) + O(\ep).
\end{equation}
From that, it is still possible to disentangle the cost and the entropy (Theorem \ref{thm:rate-cost-quadratic})
through a careful study of the behaviour of $H(\ga_\ep)$ as $\ep \to 0$. It will grant rates for $(c,\ga_\ep) - (c,\ga_0)$ of the same order $\ep$ as in last section. We start by providing a sketch of the proof: \\
Let E be a duality gap function $E = c - (\ph \oplus \ps)$. Recall that since $\ga_\ep$ has $\mu_i$ as marginals, we have that $\int E d\ga_\ep = (c,\ga_\ep)-(c,\ga_0)$, and so we will denote by $(E,\ga_\ep)$ this suboptimality.
The goal is to prove \textcolor{black}{that} $(E,\ga_\ep) \simeq \ep$ and $H(\ga_\ep) \simeq -\frac{d}{2}\ln(\ep)$, so it is natural to try to prove first $H(\ga_\ep) \simeq - \frac{d}{2}\ln(E,\ga_\ep)$, or in terms of entropy power, $N_{2d}(\ga_\ep) \simeq \sqrt{(E,\ga_\ep)}$. But thanks to the information encoded in inequality~\eqref{eq:premier_ordre_EOT}, it is enough to show an inequality of the type $N_{2d}(\ga_\ep) \leq C\sqrt{(E,\ga_\ep)}$. \\
In fact, as explained in detail in Section \ref{sec:low-bound}, this kind of inequality is not specific to $\ga_\ep$ and it comes from the following fact: the contact set $\Si := \{\textcolor{black}{(x,y) \in \mathbb{R}^d\times \mathbb{R}^d} / E(x,y) = 0\}$ has d dimensions less than the \textcolor{black}{ambient} space $\R^{2d}$ and $(E,\ga)$ quantifies the average distance of $\ga$ to $\Si$ because E has a \emph{quadratic detachment} in Minty's coordinates (see definitions \ref{def:quadratic-detachment} and \ref{def:Minty}). As a matter of fact, $W_2^2(\ga,\ga_0)$ also quantifies this distance, because $\ga_0$ is supported on $\Si$. In Proposition \ref{prop:mino-ent-quadratic} we obtain results of the form
\begin{equation}\label{eq:explication-mino-entropy-was}
    N_{2d}(\ga) \leq C \sqrt{(E,\ga)} \quad \emph{and} \quad N_{2d}(\ga) \leq C W_2(\ga,\ga_0)
\end{equation}
or equivalently in term of entropy, 
\begin{equation}\label{eq:explication-mino-entropy-E}
        H(\ga) \geq -\frac{d}{2}\ln(E,\ga) + C \quad \emph{and} \quad H(\ga) \geq -\frac{d}{2}\ln(W_2^2(\ga,\ga_0)) + C.
\end{equation} These inequalities are the key for all our rates of convergence of $(E,\ga_\ep)$, $W_2^2(\ga_\ep,\ga_0)$ and $H(\ga_\ep)$ detailed in Section \ref{sec:rates-quad}.

\subsection{Lower bounds on the entropy} \label{sec:low-bound}
The results of this section are completely independent of the optimality of $\ga_\ep$, and even independent of the fact that $\ga_\ep$ transports $\mu_0$ on $\mu_1$. Hence in the following, we will replace $\ga_\ep$ by a generic $\ga \in \P_{ac}(\R^{2d})$.
The results presented in this section are lower bounds on its entropy $H(\ga)$. These inequalities rely on the comparison of $\ga$ with a Gaussian along directions transverse to the support of a generic optimal plan $\ga_0$ for $\ep = 0$. Indeed we know that any optimal plan $\ga_0$ is supported on $\Sigma = \{E=0\}$ which is \textcolor{black}{approximately} a submanifold of (co)dimension $d$. We will then show that $\int E d\ga$ as well as $W^2_2(\ga,\ga_0)$ control the variance of $\ga$ along transverse directions to $\Sigma$. Finally since under variance constraint the Gaussian has the smallest entropy we will conclude with the wanted lower bounds of equations \eqref{eq:explication-mino-entropy-was} and \eqref{eq:explication-mino-entropy-E}. The same strategy will work in the case of twisted costs. However in the quadratic case the proof requires only a global change of variable (Minty's trick) which makes the arguments clearer. We start with the fundamental property of a function that allows for estimates like \eqref{eq:explication-mino-entropy-was} or \eqref{eq:explication-mino-entropy-E}, the \emph{quadratic detachment}.

\begin{definition}[Quadratic detachment]\label{def:quadratic-detachment}
    Let $G : \R^d \times \R^d \to \mathbb{R}_+$ be a function. We say that $G$ has a \emph{quadratic detachment} if for any $u \in \R^d$ it exists $v_u \in \R^d$ such that
    \begin{equation}
       \forall (u,v) \in \R^{2d}  \quad G (u,v) \geq \frac{1}{2}\Vert v - v_u \Vert^2.
    \end{equation}
\end{definition}

This quadratic detachment property allows for a bound on the differential entropy \textcolor{black}{$H(\ga):=H(\ga\mid\mathcal{H}^{2d})$} of any plan $\ga \in \P_{ac}(\R^{2d})$:

\begin{proposition}\label{prop:detachment-entropy}
    Let $G$ be a function on $\R^d \times \R^d$ and $\ga \in \P_{ac}(\R^{2d})$ and denote $C_d := -\frac{d}{2}\ln(\frac{4 \pi e}{d})$. If G has a quadratic detachment 
    \begin{equation}
        H(\ga) \geq -\frac{d}{2}\ln\left(\int G d\ga \right) + H(\ga_1) + C_d
    \end{equation} 
\textcolor{black}{where} $\ga_1$ is the projection of $\ga$ on the \textcolor{black}{first d coordinates}, \textcolor{black}{i.e.} the first marginal of $\ga$, and where H is the differential entropy \textcolor{black}{($H(\ga) = H(\ga\mid\mathcal{H}^{2d})$ and $H(\ga_1) = H(\ga_1\mid\mathcal{H}^d)$)}.
\end{proposition}

\begin{proof}
If $H(\ga) = +\infty$ or $\int G d\ga = +\infty$ there is nothing to prove. We assume that $H(\ga) < +\infty$ and $\int G d\ga < +\infty$.
Let $\ga = \ga_1 \otimes \ga^u$ be the disintegration of $\ga$ \textcolor{black}{with respect to the projection on the first d coordinates}. $\int G d\ga < +\infty$ implies that $\ga^u$ has a finite moment of order $2$, $\ga_1$ almost everywhere, because $G$ has a quadratic detachment. \textcolor{black}{So its differential entropy $H(\ga^u)$ is well defined. Let $g_d$ denote the standard d-dimensional normal distribution. We will use the link between the differential entropy and the entropy with respect to $g_d$ \cite[Annexe A]{leonard} as well as the additivity property of the entropy (see equation (A.8) of \cite{leonard} or \cite[Theorem 2.4]{leonard2014some}) to express $H(\ga)$ with $H(\ga^u)$:
    \begin{align}\label{eq:additivity}
    \begin{split}
         H(\ga) &= H(\ga|g_{2d}) -d\ln(2\pi) - \frac{1}{2}\int\Vert x\Vert^2d\ga(x)
         \\
         &=  H(\ga_1|g_d) + \int H(\ga^u |g_d) d\ga_1(u) - d\ln(2\pi) - \frac{1}{2}\int(|u|^2 + |v|^2) d\ga^u(v) d\ga_1(u) \\
         &= H(\ga_1|g_d) - \frac{d}{2}\ln(2\pi) -  \frac{1}{2}m_2(\ga_1) + \int \left( H(\ga^u |g_d) -  \frac{d}{2}\ln(2\pi) -  \frac{1}{2}m_2(\ga^u) \right) d\ga_1(u) \\
         &= H(\ga_1) + \int H(\ga^u) d\ga_1(u).
         \end{split}
    \end{align}
}
       
    However under variance constraint the \textcolor{black}{Gaussian} with independent coordinates has the smallest entropy, thus
        $H(\ga^u) \geq H( \mathcal{N}(0,\frac{\text{Var}(\ga^u)}{d} I_d)) = -\frac{d}{2}\ln\left( \text{Var}(\ga^u) \right) -\frac{d}{2}\ln(\frac{2 \pi e}{d})$.
Moreover, the average of a random variable minimizes the average square distance, so $
\text{Var}(\ga^u) \leq \int \| v - v_u \|^2 d\textcolor{black}{\ga^u}(v) \leq 2 \int G(u,v) d\textcolor{black}{\ga^u}(v)$ where the last inequality holds because G has a quadratic detachment. So we have
    \begin{equation}
        H(\ga) \geq H(\ga_1) - \int \frac{d}{2}\ln\left( 2 \int G(u,v)d \ga^u(v) \right) d\ga_1(u) -\frac{d}{2}\ln(\frac{2 \pi e}{d}).
    \end{equation}

Now by concavity of the logarithm, exchanging it with the integral against $\ga_1$ gives an even lower quantity, hence we obtain the desired inequality.
\end{proof}

This result shows a clear path to get lower bounds on the entropy of $\ga_\ep$: typically, it is enough to show that the duality gap has a quadratic detachment to show inequality \eqref{eq:explication-mino-entropy-E}. However, E has no quadratic detachment in the classic $(x,y)$ coordinates, but in transverse coordinates, that we call \emph{Minty's coordinates} (\cite{minty1962trick}).

\begin{definition}[Minty's coordinates] \label{def:Minty}
    We wil call \emph{Minty's coordinates} the coordinates $(u,v) = \frac{1}{\sqrt{2}}(x+y,x-y) $. The change of variable $(x,y) \mapsto (u,v)$ is
 an isometry and for any convex function $f$, the graph of the subdifferential of $f$ in the $(x,y)$-coordinates corresponds to the graph of a 1-Lipschitz function defined on the whole space  in the $(u,v)$-coordinates (see \cite[Theorem 12.12, Theorem 12.25]{RockWets1998variationalanalysis}).
Moreover, for $\ga \in \P(\R^{2d})$, we will denote by $\hat{\ga}$ the pushforward of $\ga$ through this change of variable and call \emph{Minty's decomposition} of $\ga$ the disintegration of $\hat{\ga}$ with respect to the orthogonal projection onto the \textcolor{black}{$u$} coordinate $\hat{\ga} = \hat{\mu} \otimes \hat{\ga}^u$.
\end{definition}

The quadratic detachment of E in \textcolor{black}{these} coordinates follows from the slightly more general inequality \eqref{eq:minty-trick}, known as Minty's trick. It can be generalized for twisted costs (see Lemma \ref{lem:mtw-minty-trick}). In the quadratic case, it is somehow folklore and we recall the proof, but for more general cost, it was proved in \cite{carlier2022convergence}, following arguments of
\cite{mccann2012rectifiability}. We also prove the quadratic detachment of the function $d(.,\Si)^2$ that we will use to bound the entropy with the Wasserstein distance to an optimizer.

\begin{lemma} \label{lem:minty}
Let $\mu_0,\mu_1 \in \P_2(\R^d)$ and $E := c - (\ph \oplus \ps)$ the duality gap of the quadratic optimal transport problem from $\mu_0$ to $\mu_1$. Denote by $d(z,\Si)$ the distance of a point $z \in \R^{2d}$ to the set $\Si = \{E=0\}$. In Minty's coordinates, the functions $(u,v) \mapsto E(u,v)$ and $(u,v) \mapsto d((u,v),\Si)^2$ have quadratic detachment.\footnote{\textcolor{black}{Here}, by a slight abuse of notation, $E$ denotes $E\circ \Ph^{-1}$ where $\Ph: (x,y) \mapsto (u,v)$ is Minty's change of variable (definition \ref{def:Minty}).}
\end{lemma}

\begin{proof}
    First note that in the quadratic case $E$ rewrites as $E(x,y) = f(x) + f^\ast (y) - \langle x,y \rangle $ where $f$ is a convex function and $f^*$ is its Legendre-Fenchel conjugate. Thus for $(x,y),(x',y') \in \mathbb{R}^d \times \mathbb{R}^d$ we have
    \begin{equation} \label{eq:minty-intermediate}
    \begin{split}
        E(x,y) + E(x',y') &=f(x) + f^\ast (y) - \langle x,y \rangle + f(x') + f^\ast (y') - \langle x',y' \rangle \\
        & \geq \langle x,y'\rangle - f^\ast (y') + f^\ast (y) - \langle x,y \rangle + \langle x',y \rangle - \textcolor{black}{f^\ast(y)} + f^\ast(y')- \langle x',y' \rangle\\
        & \geq \langle x-x',y'-y \rangle.
    \end{split}
    \end{equation}
    In Minty's coordinates $(u,v) = \frac{1}{\sqrt{2}}(x+y,x-y)$ we can see $E$ as a function of $(u,v)$. In these coordinates, inequality \eqref{eq:minty-intermediate} becomes 
    \begin{equation}\label{eq:minty-trick}
        E(u,v) + E(u',v') \geq \frac{1}{2} (\Vert v' - v \Vert^2 - \Vert u' - u \Vert^2).
    \end{equation}
    The set $\Si = \{\textcolor{black}{(x,y) \in \mathbb{R}^d\times \mathbb{R}^d} / E(x,y)=0\}$ corresponds to the subdifferential of the convex function $f$. Hence \textcolor{black}{in Minty's coordinates} $(u,v)$, it corresponds to the graph in u of a \textcolor{black}{1-Lipschitz} function that is defined on the whole space $\R^d$ (see \cite[Theorem 12.12, Theorem 12.25]{RockWets1998variationalanalysis}). So for any $u \in \R^d$, it  exists a unique $ v_u \in \R^d$ such that  $E(u,v_u) = 0$. And thanks to Minty's trick \eqref{eq:minty-trick} we have that $E(u,v) \geq \frac{1}{2} \vert v - v_u \vert^2$. \\
    It remains to show that the distance to the graph of the \textcolor{black}{1-Lipschitz} function $u \mapsto v_u$ has a quadratic detachment. Let $u,u',v \in \R^d$. From the inequality $ \|a\|^2 + \|b\|^2 \geq \frac{1}{2} \|a+b\|^2$ applied to $a = v_u - v_{u'}$ and $b = v_{u'} - v$, we get 
    \[
    \|v_u - v_{u'}\|^2 + \|v_{u'} - v\|^2 \geq \frac{1}{2} \| v_u - v\|^2.
    \]
    And we can bound $\|v_u - v_{u'}\|^2$ by $\|u-u'\|^2$ because $u \mapsto v_u$ is \textcolor{black}{1-Lipschitz}. Hence we obtain
    \[
    \|u - u'\|^2 + \|v_{u'} - v\|^2 \geq \frac{1}{2} \| v_u - v\|^2.
    \]
    But the \textcolor{black}{left-hand side} is the \textcolor{black}{(squared) distance} from $(u,v)$ to $(u',v_{u'})$ and the inequality holds for any $u'$, so 
    \[
   \textcolor{black}{d((u,v),\Si)^2} \geq \frac{1}{2} \| v_u - v\|^2.
    \]
\end{proof}

Now that we know that E and $d(.,\Si)^2$ have quadratic detachment, we can use Proposition \ref{prop:detachment-entropy} to bound the entropy of $\ga$ from \textcolor{black}{below} with $\int E d\ga$ and $W_2^2(\ga,\ga_0)$.

\begin{proposition}\label{prop:mino-ent-quadratic}
    Suppose that the cost is quadratic. Let $\textcolor{black}{\mu_0,\mu_1} \in \P_2(\R^d)$, $(\ph, \ps)$ be the associated Kantorovich potentials and  $E : = c - (\ph \oplus \ps)$ be a duality gap function. Let $\ga_0$ be an optimal transport plan from $\mu_0$ to $\mu_1$. Then, for any $\ga \in \P_{ac}(\R^{2d})$
\begin{equation} \label{eq:low-bound-E}
        H(\ga) \geq -\frac{d}{2}\ln\left(\int E d \ga\right) + H(\hat{\mu}) + C_d
    \end{equation}
    and 
    \begin{equation} \label{eq:low-bound-W}
        H(\ga) \geq -\frac{d}{2}\ln W^2_2(\ga,\ga_0) + H(\hat{\mu}) + C_d
    \end{equation}
where $\hat{\ga} = \hat{\mu}\otimes \hat{\ga}^u$ is the Minty's decomposition of $\ga$ and $C_d := -\frac{d}{2}\ln(\frac{4 \pi e}{d})$ (see definition \ref{def:Minty}). \\
If we denote by $\si_{\ga}(X+Y)$ the standard deviation of $X+Y$ when the law of $(X,Y)$ is $\ga$, we have
    \begin{equation} \label{eq:lower-bounds}
        N_{2d}(\ga) \leq \frac{\si_{\ga}(X+Y)}{d} \sqrt{\int E d\ga} \qquad  N_{2d}(\ga) \leq \frac{\si_{\ga}(X+Y)}{d} W_2(\ga,\ga_0).
    \end{equation}
   
\end{proposition}

\begin{proof}
By Lemma \ref{lem:minty}, the duality gap E and the square distance $d(.,\Si)^2$ have quadratic detachment. So we can apply Proposition \ref{prop:detachment-entropy} on both to obtain
\begin{align}
    H(\ga) \geq -\frac{d}{2}\ln\left(\int E d \ga\right) + H(\hat{\mu}) + C_d  \\ H(\ga) \geq -\frac{d}{2}\ln\left(\int d(\bm{x},\Si)^2 d \ga(\bm{x})\right) + H(\hat{\mu}) + C_d.
\end{align}
The first inequality is exactly \eqref{eq:low-bound-E} and the second implies \eqref{eq:low-bound-W} because an optimal plan $\ga_0$  is concentrated on a contact set $\Si =\{\textcolor{black}{(x,y) \in \mathbb{R}^d \times \mathbb{R}^d} / E(x,y) = 0\}$ by optimality and so $W_2^2(\ga,\ga_0) \geq \int d(\bm{x},\Si)^2 d\ga(\bm{x})$. \\
The second part is a consequence and we provide only the proof for $\int E d\ga$ because the proof for $W_2^2(\ga_\ep,\ga_0)$ is stricly similar. Taking the exponential of \eqref{eq:low-bound-E}, we get
\[
N_{2d}(\ga)^2 \leq \frac{2}{d} N_d(\hat{\mu}) \int E d\ga.
\]
Since the \textcolor{black}{Gaussian} minimizes the entropy at fixed variance, the entropy power of $\hat{\mu}$ is \textcolor{black}{smaller} than the one of a \textcolor{black}{Gaussian} of same variance, that is $N_d(\hat{\mu}) \leq \frac{\text{Var}(\hat{\mu})}{d}$. Moreover $\text{Var}(\hat{\mu}) = \frac{1}{2}\text{Var}_\ga(X+Y)$ because $\hat{\mu}$ is the law of the first marginal of $\ga$ in the coordinates $(u,v) = \frac{1}{\sqrt{2}}(x+y,x-y)$. This is enough to finish the proof of \eqref{eq:lower-bounds}.
\end{proof}

We have established the bounds \eqref{eq:lower-bounds} on a general plan $\ga$. It is now time to apply these results to the entropic plan $\ga_\ep$ under hypothesis $\bm{(H2)}$ in order to obtain rates of convergence of the suboptimality $(c,\gamma_\ep) -(c,\gamma_0)$. We will also quantify the weak convergence of $\ga_\ep$ towards $\ga_0$ in $W_2$ distance.

\subsection{Rates of convergence}\label{sec:rates-quad}
In the sequel of the section, we use the set of assumptions $\bm{(H2)}$ that imply in particular that the optimizer $\ga_0$ is unique. The main result of this section states that the suboptimality $(c,\gamma_\ep) -(c,\gamma_0)$ converges to $0$ at a speed of order $\ep$. We also derive the same speed for $W^2_2(\gamma_\ep,\gamma_0)$ and this matches in both cases the rate found for \textcolor{black}{Gaussians}. 

The proofs of the Theorems require an \textcolor{black}{upper bound} on the rate of convergence of $OT_\ep$ towards $OT_0$. The key point being that the convergence rate is of the form
\begin{equation}
    \ep H(\ga_\ep) \leq OT_\ep - OT_0 \leq -\frac{d}{2}\ep \ln(\ep) + O(\ep).
\end{equation}
This result has been obtained recently for general costs under some regularity conditions \cite{carlier2022convergence}  and in the quadratic case under moments conditions \cite{nutz2022rate}. The following Lemma \textcolor{black}{makes explicit} the dependence of the bounded term $O(\ep)$ found in the latter.

\begin{lemma}\label{lem:otep-ot-constant-in-rate}
    Assume as in hypothesis $\bm{(H2)}$ that the cost is quadratic, that $\mu_0,\mu_1 \in \mathcal{P}_{2+\delta}(\mathbb{R}^d)$ for some $\delta > 0$ and that $H(\mu_i) < + \infty$.  Then, as $\ep \to 0$
    \begin{equation}\label{eq:rate-ote-oto}
        OT_\ep - OT_0 \leq -\frac{d}{2}\ep \ln(\ep) + C\ep,
    \end{equation}
    for a constant $C$ depending on d, $m_{2+\de}(\mu_i)$ and $H(\mu_i)$. 
\end{lemma}

\begin{proof}
    The constant $C$ introduced in \cite[Corollary 3.14]{nutz2022rate} is (up to a numerical constant) the quantization constant of the optimal transport $\gamma_0$.
    \textcolor{black}{The quantization constant at a rate $\alpha$ of a measure $\mu$ is a constant C such that for any $n \geq 0$ there is a measure $\mu_n$ which is the mean of $n$ Dirac masses such that $W_2(\mu,\mu_n) \leq Cn^{-\alpha}$.
    Here the quantization happens with a rate $\frac{1}{d}$ and thus the quantization constant} allows for the following upper bound
    \begin{equation}
        OT_\ep - OT_0 \leq -\frac{d}{2}\ep \ln(\ep) + (H(\mu_0)+H(\mu_1))\ep + C\ep.
    \end{equation}
    However by Minty's trick \cite[Lemma 3.13]{nutz2022rate}  this quantization constant is tied to the quantization constant of a measure admitting a moment of order $2+\delta$. In that case \cite[Corollary 6.7]{siegfriedHarald2000quantization} gives an explicit bound
    \begin{equation}
        C \leq C'\left(m_{2+\delta}(\ga_0) + d\right)
    \end{equation}
    where $C'$ is independent of dimension.
    We conclude by using triangular inequality to upper bound  $m_{2+\delta}(\ga_0)$ by $2^{2+\delta}(m_{2+\delta}(\mu_0)+m_{2+\delta}(\mu_1))$. 
\end{proof}

With the inequality \eqref{eq:rate-ote-oto} giving an upper bound on $H(\ga_\ep)$ and the lower bound \eqref{eq:low-bound-E} found in the previous section, we are now able to show that $(c,\ga_\ep) - (c,\ga_0) = \Th(\ep)$:

\begin{theorem}\label{thm:rate-cost-quadratic}
    Suppose that the cost is quadratic. Let $\mu_0,\mu_1 \in \mathcal{P}_{2+\delta,ac}(\mathbb{R}^d)$ for some $\delta >0$. And further assume that their entropy is finite. Then in the limit $\ep \to 0$,
\begin{equation}
    (c,\ga_\ep) - (c,\ga_0) = \Th(\ep), \quad H(\ga_\ep) = - \frac{d}{2} \ln(\ep) + O(1) \quad \emph{and} \quad OT_\ep = OT_0 - \frac{d\ep}{2} \ln(\ep) + O(\ep).
\end{equation}
More precisely, there exist $c,C>0$, depending on $m_{2+\de}(\mu_i)$, $H(\mu_i)$ and the dimension d, such that
    \begin{equation}
        c \ep \leq \int E d \ga_\ep = (c,\ga_\ep) - (c,\ga_0) \leq C \ep.
    \end{equation}
\end{theorem}

\begin{proof}
 Set $(E,\ga_\ep) := \int E d\ga_\ep$, and consider d, $H(\mu_i)$ and $m_{2+\de}(\mu_i)$ as fixed. Lemma \ref{lem:otep-ot-constant-in-rate} grants the following upper bound on the regularized problem for $\ep$ small enough
\begin{equation}\label{eq:majo-quant-rate}
    (E,\ga_\ep) + \ep H(\gamma_\ep) \leq - \frac{d}{2} \ep \ln(\ep) + C\ep.
\end{equation}
On the other hand, from Proposition \ref{prop:mino-ent-quadratic}, we know that $N_{2d}(\ga_\ep) \leq \frac{\si_{\ga_\ep}(X+Y)}{d} \sqrt{(E,\ga_\ep)}$.
But we have also $\si_{\ga_\ep}(X+Y) \leq \si_{\ga_\ep}(X) + \si_{\ga_\ep}(Y) = \sqrt{m_2(\mu_0)} + \sqrt{m_2(\mu_1)}$ because the marginals of $\ga_\ep$ are $\mu_0$ and $\mu_1$. And by Hölder inequality, $m_2(\mu) \leq m_{2+\de}(\mu)^{\frac{2}{2+\de}}$.
So $N_{2d}(\ga_\ep) \leq C' \sqrt{(E,\ga_\ep)}$ or equivalently, by taking the logarithm,
\begin{equation}\label{eq:minty-wass-1}
	H(\ga_\ep) \geq -\frac{d}{2}\ln\left(E,\ga_\ep\right) + C''.
\end{equation}
Now by combining \eqref{eq:minty-wass-1} with \eqref{eq:majo-quant-rate} we have
\begin{equation}
    \frac{(E,\ga_\ep)}{\ep}   - \frac{d}{2} \ln\left(\frac{(E,\ga_\ep)}{\ep}\right)  \leq C'''.
\end{equation}
The function $x \mapsto x- d/2 \ln(x)$ goes to $+\infty$ near $0$ and $+\infty$. Thus $\frac{(E,\ga_\ep)}{\ep} = \Th(1)$, or more precisely, there exist constants $c,C>0$, depending on d, $m_{2+\de}(\mu_i)$ and $H(\mu_i)$, such that $c\ep \leq (E,\ga_\ep) \leq C \ep$. Injecting $(E,\ga_\ep) = \Th(\ep)$ in inequalities \eqref{eq:majo-quant-rate} and \eqref{eq:minty-wass-1} gives respectively $H(\ga_\ep) \leq - \frac{d}{2} \ln(\ep) + O(1)$ and $H(\ga_\ep) \geq - \frac{d}{2} \ln(\ep) + O(1)$ which concludes the proof.
\end{proof}

Now that we have proven that the suboptimality $(c,\ga_\ep) - (c,\ga_0)$ has speed $\ep$, we want to do the same for the the Wasserstein distance $W_2^2(\ga_\ep,\ga_0)$. Inequality \eqref{eq:low-bound-W} will give a lower bound on $W_2^2(\ga_\ep,\ga_0)$, but we need also an upper bound. If the Brenier map is Lipschitz, the following Lemma provides it.\footnote{This Lemma and its proof \textcolor{black}{were} suggested to us by P. Pegon.}

\begin{lemma}\label{lem:minty-wasserstein-upper-bound}
    For the quadratic cost, suppose that the Brenier map $T=\nabla f$ is \textcolor{black}{$L$-Lipschitz}. Then
    \begin{equation}
     W^2_2(\gamma_\ep,\gamma_0)\leq \int \Vert y - T(x) \Vert^2d\gamma_\ep\leq 2L\left((c,\gamma_\ep) - (c,\gamma_0)\right).
    \end{equation}
\end{lemma}

\begin{proof}
    Let $S(x,y) = (x,T(x))$ with $T = \nabla f$ and $f$ \textcolor{black}{a Brenier convex function associated with $\gamma_0$}. Then $S$ is a transport map from $\gamma_\ep$ to $\gamma_0$ thus $W^2_2(\gamma_\ep,\gamma_0) \leq \int \Vert y- T(x) \Vert^2 d\gamma_\ep$. We now use an inequality proved by Berman in \cite{berman2020stability} and Li and Nochetto in \cite{liNochetto2020stability}, who have built upon an earlier work of Gigli \cite{gigli2011mintytrick}. \textcolor{black}{The inequality states} that whenever the Brenier map is \textcolor{black}{$L$-Lipschitz} we have $\Vert y - T(x) \Vert^2 \leq 2L E$ with  $E = c-(\ph \oplus \ps)$ the duality gap.  Combining these two inequalities grants the result, since $\int E d\ga_\ep = (c,\ga_\ep) - (c,\ga_0)$.
\end{proof}

The hypothesis that the Brenier map is Lipschitz might seem abstract. However, note that Caffarelli's regularity theory (\cite{caffarelli1992regularity},\cite{caffarelli2000fkgmonotonicity}) ensures it holds under some simple regularity conditions on the marginals.
For example as soon as the marginals have Hölder densities bounded away from zero on their supports and the latter are smooth and uniformly convex. In a different direction, recent works \cite{sinhoaram2022cafreg} ensure regularity of the transport map between log-concave measures under a variance constraint.
\begin{theorem}\label{thm:mino-wass-gen}
    Suppose that the cost is quadratic. Let $\mu_0,\mu_1 \in \mathcal{P}_{2+\delta,ac}(\mathbb{R}^d)$ for some $\delta >0$. And further assume that their entropy is finite and that the Brenier map $T=\nabla f$ is Lipschitz. Then $W_2^2(\gamma_\ep,\gamma_0) = \Theta(\ep)$. More precisely:
    \begin{enumerate}
        \item if $\mu_0,\mu_1 \in \mathcal{P}_{2+\delta,ac}(\mathbb{R}^d)$ with finite entropy, then there \textcolor{black}{exists} $c>0$, depending on the moments $m_{2+\de}(\mu_i)$, the entropies $H(\mu_i)$ and the dimension d, such that, as $\ep \to 0$
        \begin{equation}\label{eq:mino_ent_power}
        c \ep \leq W^2_2(\gamma_\ep,\gamma_0),
        \end{equation}
        \item if moreover the Brenier map is \textcolor{black}{L-Lipschitz}, then there \textcolor{black}{exists} $C>0$ depending on d, $m_{2+\de}(\mu_i)$, $H(\mu_i)$ and L, such that, as $\ep \to 0$
        \begin{equation}\label{eq:mino_ent_power2}
        c \ep \leq W^2_2(\gamma_\ep,\gamma_0) \leq C\ep.
    \end{equation}
    \end{enumerate} 
\end{theorem}

\begin{proof}
For the first part, \textcolor{black}{let us} recall that from Proposition \ref{prop:mino-ent-quadratic}, we have $N_{2d}(\ga_\ep) \leq \frac{\si_{\ga_\ep}(X+Y)}{d} \textcolor{black}{W_2(\ga_\ep,\ga_0)}$. As in the proof of Theorem \ref{thm:rate-cost-quadratic}, \textcolor{black}{we can bound} $\si_{\ga_\ep}(X+Y)$ with the moments $m_{2+\de}(\mu_i)$ and so we obtain
\begin{equation}\label{eq:power-W2}
    N_{2d}(\ga_\ep) \leq C(m_{2+\de}(\mu_i),d) W_2(\ga_\ep,\ga_0).
\end{equation}
But on the other hand, from Lemma \ref{lem:otep-ot-constant-in-rate}, we \textcolor{black}{know} that $H(\ga_\ep) \leq -\frac{d}{2} \ln(\ep) + C(m_{2+\de}(\mu_i),H(\mu_i),d)$. Or equivalently,
\begin{equation}\label{eq:power-ep}
    N_{2d}(\ga_\ep) \geq C(m_{2+\de}(\mu_i),H(\mu_i),d) \sqrt{\ep}.
\end{equation}
So, combining \textcolor{black}{inequalities} \eqref{eq:power-W2} and \eqref{eq:power-ep}, we get $W_2^2(\ga_\ep,\ga_0) \geq c \ep$ for some constant c depending on d, $H(\mu_i)$ and $m_{2+\de}(\mu_i)$. \\
The second part is a direct combination of Lemma \ref{lem:minty-wasserstein-upper-bound} from which  $W^2_2(\gamma_\ep,\gamma_0) \leq 2L\left((c,\gamma_\ep) - (c,\gamma_0)\right)$
and Theorem \ref{thm:rate-cost-quadratic} which states that $(c,\gamma_\ep) - (c,\gamma_0) \leq C(m_{2+\delta}(\mu_i),H(\mu_i),d) \ep$.
\end{proof}

\begin{remark}
Note that, as explained in \cite[Proposition 4.5]{carlier2022convergence}, the method of Lemma \ref{lem:minty-wasserstein-upper-bound} allows to control also the $L^2$-gap between $T_\ep$, the barycentric projection of $\ga_\ep$ \textcolor{black}{defined by $T_\ep(x) = \int y d\ga^x_\ep(y)$ where $\mu_0 \otimes \ga^x_\ep$ is the disintegration of $\ga_\ep$ with respect to the projection onto the first $d$ coordinates}, and the Brenier map $T=\nabla f$.
\[ ||T_\ep - T||_{L^2(\mu)}^2 \leq \int |y-T(x)|^2 d\ga_\ep(x,y) \leq 2L \int E d\ga_\ep \leq C \ep.
\] 
Note that, contrary to the estimate on $W_2(\ga_\ep,\ga_0)$, this estimate is not sharp for Gaussian measures where $||T_\ep - T||_{L^2(\mu)}^2$ can be computed to be of order $\ep^2$ \cite{NIPS20}. In a similar fashion we get a control over the distance between the support of $\gamma_\ep$ and the graph of the optimal transport
\begin{equation}
    \int d(y,\text{supp}(\gamma_0))^2 d\gamma_\ep \leq \int |y-T(x)|^2 d\ga_\ep(x,y) = O(\ep).
\end{equation}
\end{remark}

Hence, for the quadratic case, under technical assumptions, we \textcolor{black}{have obtained} generic rates of convergence when the marginals $\mu_i$ have finite differential entropy: 
\begin{equation}
    (c,\ga_\ep) = (c,\ga_0) + \Th(\ep), \qquad H(\ga_\ep) = - \frac{d}{2}\ln(\ep) + O(1), \qquad W_2(\ga_\ep,\ga_0) = \Th(\sqrt{\ep}).
\end{equation}
Can we generalize it to a broader class of costs? This is the objective of next section.

\section{Rates for infinitesimally twisted costs}\label{sec:general-rates}
In this section, we use the \textcolor{black}{set of assumptions} $\bm{(H3)}$ that is adapted to twisted cost functions. As in the quadratic case, we begin by estimates on the entropy based on quadratic detachment.
\subsection{Lower bound on the entropy}
\textcolor{black}{In the previous section} the key ingredient used to lower bound the entropy was the lower bound involving the variance $H(\ga) \geq H(\mathcal{N}(0,\text{Var}(\ga)/d))$. In fact we used a finer result using a disintegration of $\ga$ in two components, one of which had bounded variance. In order to extend the preceding results to more general costs our goal is to have a \emph{local} version of the Proposition \ref{prop:mino-ent-quadratic} that states a precise lower bound for $H(\ga)$. Indeed, imagine that $E \geq 0$ is a function such that there is a direction along which $E$ grows quadratically locally around $\Sigma = \{ E = 0 \}$. For any measure $\ga$, it is then possible to bound the variance of $\ga$ along that same direction which \textcolor{black}{in turn gives} a bound on the entropy of $\ga$ in a similar fashion to Lemma \ref{lem:minty}.  In order to state this idea formally we introduce the concept of \emph{local} quadratic detachment for a positive function.

\begin{definition}[Local quadratic detachment]\label{def:loc-quadratic-detachment}
    Let $X\subset \mathbb{R}^d$ be a compact set and $E : X \times X \to \mathbb{R}_+$ continuous. Assume that $\Sigma = \{E=0\} \neq \emptyset$. \textcolor{black}{We say} that $E$ has a \emph{local quadratic detachment} of parameters $((U_i)_i,(\Ph_i)_i,\kappa)$ where $\kappa> 0$, $(U_i)_i$ is a finite open covering  of $\Sigma$, and $(\Ph_i)_i$ a family of volume preserving affine functions $\Ph_i: \bm{x}\mapsto (u,v)\in \mathbb{R}^d \times \mathbb{R}^d$  such that for all $(u,v),(u,v') \in \Ph_i(U_i)$ 
    \begin{equation}
        E (u,v) +E (u,v') \geq \kappa\Vert v-v' \Vert^2
    \end{equation}
    \textcolor{black}{where} $E$ denotes $E \circ \Ph_i^{-1}$ by a slight abuse of notation.
\end{definition}
This local quadratic detachment property is a direct generalisation of the notion of quadratic detachment (definition \ref{def:quadratic-detachment}). And it happens that the same result holds: the entropy of $\ga$ is lower bounded by the log of the integral of a function having a local quadratic detachment. The idea of proof is the same as in the quadratic case with some \textcolor{black}{additional} technical issues associated to the locality of the quadratic detachment.
\begin{proposition}\label{prop:lower-bound-ent-local-detachment}
    Let $X\subset \mathbb{R}^d$ be a compact set and $E$ be a continuous function on $X\times X$ that has a local quadratic detachment with parameters $((U_i)_{i=1}^N,(\Ph_i)_{i=1}^N,\kappa)$.  Set $R =X \times X \setminus \bigcup_i U_i$. There exists a constant $C$ depending on $X,\Vert \Ph_i \Vert_{op}$ and $N$ such that for any $\ga \in \mathcal{P}(X\times X)$ we have 
    \begin{equation}
        H(\ga) \geq -\frac{d}{2}\ln\left(\int E(\bm{x}) d\ga\right) - \frac{d}{2}\ln\left(\frac{4\pi e}{\kappa d}\right) + \frac{d}{2}\ln\left(1 \wedge \inf_{\bm{x}\in R}E(\bm{x})\right) + C.
    \end{equation}
\end{proposition}

\begin{proof}
    First if $H(\ga) = + \infty$ or $\int E d\ga = + \infty$, there is nothing to prove. We assume that $H(\ga) < + \infty$ thus $\ga$ has a density that we denote $\rho$ and $\int E d\ga < + \infty$. Let $\kappa>0$, $(U_i)_{i=1}^N$ and $\Ph_i$ be as in the definition \ref{def:loc-quadratic-detachment} of the local quadratic detachment. Let $\zeta_i$ be a partition of unity subordinate to the family $(U_i)_i$. Denote by $U = \bigcup_i U_i$. If $p_i = (\zeta_i\rho)(U) > 0$ set $\rho_i = \frac{1}{p_i} \zeta_i\rho$, else $\rho_i = 0$. Then we have the following decomposition $\rho = 1_U \sum_i p_i \rho_i + 1_{U^c} \rho$ which we can inject in the definition of the entropy
    \begin{equation}\label{eq:entropy-gamma-detachment}
    \begin{split}
        H(\ga) &= \int \rho \ln(\rho)\\
        &= \int_U \sum_i p_i \rho_i \ln(\sum_j p_j \rho_j) + \int_{U^c}\rho \ln(\rho)\\
        &= \sum_i p_i \int \rho_i \ln(\sum_j p_j \rho_j) + \int_{U^c} \rho \ln(\rho)\\
        &\geq \sum_i p_i \int \rho_i \ln(p_i \rho_i) + \int_{U^c} \rho \ln(\rho)\\
        &\geq \sum_i p_i \ln(p_i) + \sum_i p_i \int \rho_i \ln(\rho_i) + \int_{U^c} \rho \ln(\rho)\\
        &\geq \sum_i p_iH(\rho_i)- \frac{1}{e}\left(N + \mathcal{H}^{2d}(X\times X)\right)
    \end{split}
    \end{equation}
    \textcolor{black}{where} the first inequality holds because $p_i \rho_i \leq \sum_j p_j \rho_j$ and the last inequality holds because under support constraint the uniform distribution has the smallest entropy. Now for $i$ such that $\rho_i \neq 0$ set $\tau_i = (\Ph_i)_\# \rho_i$, since $\Ph_i$ is volume preserving we have $H(\tau_i) = H(\rho_i)$. We now disintegrate $\tau_i$ with respect to the projection onto the first coordinate $u$ and denote the disintegration $\tau_i = \mu_i \otimes \tau^u_i$. \textcolor{black}{Following the same method as in equations \eqref{eq:additivity}, the additivity of the entropy gives}
    \begin{equation}\label{eq:additivity-proof-detachment}
        H(\tau_i) = H(\mu_i) + \int H(\tau^u_i) d\mu_i.
    \end{equation}
    However let $u$ be in the support of $\mu_i$ and let $v$ be within the support of $\tau^u_i$. Take $v'\in \pi_2 \Ph_i(X\times X)$ such that $E(u,v') = \min_{\pi_1\Ph_i(X\times X)} E(u,.)$, which is possible by compactness of $\pi_2\Ph_i(X\times X)$ and continuity of $E$. Then by the local quadratic detachment of $E$ we have
    \begin{equation}
        E(u,v) \geq \frac{1}{2}(E(u,v')+E(u,v)) \geq \frac{\kappa}{2} \Vert v'-v\Vert^2.
    \end{equation}
    Moreover since the variance is the distance to the constants in $L^2$ space we have
    \begin{equation}
        \int E(u,v) d\tau^u_i \geq \frac{\kappa}{2} \int\Vert v'-v\Vert^2 d\tau^u_i \geq \frac{\kappa}{2} \text{Var}(\tau^u_i).
    \end{equation}
    Using this control of the variance in \textcolor{black}{conjunction} with the fact that the \textcolor{black}{Gaussian} \textcolor{black}{has the lowest} entropy under variance constraint we have
    \begin{equation}\label{eq:lower-bound-entropy-tau}
    \begin{split}
        H(\tau^u_i) &\geq - \frac{d}{2}\ln\left(\text{Var}(\tau^u_i)\right) - \frac{d}{2}\ln\left(\frac{2\pi e}{d}\right)\\
        &\geq - \frac{d}{2}\ln\left(\int E(u,v) d\tau^u_i\right) - \frac{d}{2}\ln\left(\frac{4\pi e}{\kappa d}\right).
    \end{split}
    \end{equation}
    On the other hand we know that $\mu_i$ is supported in $\pi_1\Ph_i(U_i)$ the diameter of which is smaller than a constant time that of $X$, thanks to \textcolor{black}{$\Ph_i$ being affine with bounded derivative}. Thus since on a bounded support the uniform distribution has the smallest entropy we have $H(\mu_i) \geq -\ln( \mathcal{H}^d(\pi_1\Ph_i(U_i))) = C(\Vert \Ph_i \Vert_{op},X)$.
    Now coming  back to equation \eqref{eq:additivity-proof-detachment} we have
    \begin{equation}\label{eq:mino-ent-chgt-var-detach}
    \begin{split}
        H(\rho_i) = H(\tau_i) &= H(\mu_i) + \int H(\tau^u_i) d\mu_i \\
        &\geq- \frac{d}{2}\int \ln\left(\int E(u,v) d\tau^u_i\right) d\mu_i- \frac{d}{2}\ln\left(\frac{4\pi e}{\kappa d}\right) + C(\Vert \Ph_i \Vert_{op},X)\\
        &\geq - \frac{d}{2}\ln\left(\int E(u,v) d\tau_i\right)- \frac{d}{2}\ln\left(\frac{4\pi e}{\kappa d}\right) + C(\Vert \Ph_i \Vert_{op},X) \\
        &= - \frac{d}{2}\ln\left(\int E(\bm{x}) d\rho_i\right)- \frac{d}{2}\ln\left(\frac{4\pi e}{\kappa d}\right) + C(\Vert \Ph_i \Vert_{op},X),
    \end{split}
    \end{equation}
    where the second to last inequality holds by concavity of the logarithm and the last inequality holds by \textcolor{black}{defintion of the pushforward}. We now multiply equation \eqref{eq:mino-ent-chgt-var-detach} by $p_i$ and sum over $i$ in order to combine it with the lower bound term in equation \eqref{eq:entropy-gamma-detachment}
    \begin{equation}
    \begin{split}
        H(\ga) &\geq \sum_i p_iH(\rho_i)- \frac{1}{e}\left(N + \mathcal{H}^{2d}(X\times X)\right)\\
        &\geq \sum_i p_i \left(- \frac{d}{2}\ln\left(\int E(\bm{x}) d\rho_i\right)- \frac{d}{2}\ln\left(\frac{4\pi e}{\kappa d}\right) + C(\Vert \Ph_i \Vert_{op},X) \right) - \frac{1}{e}\left(N + \mathcal{H}^{2d}(X\times X)\right)\\
        &\geq -\frac{d}{2}\rho(U)\ln\left(\int_U E(\bm{x}) d\rho\right) - \frac{d}{2}\ln\left(\frac{4\pi e}{\kappa d}\right)   + C(N,\Vert \Ph_i \Vert_{op},X),
    \end{split}
    \end{equation}
    where we used the concavity of the logarithm one last time. It remains to prove that there is a constant such that $-\rho(U)\ln\left(\int_U E d\rho\right) \geq -\ln\left(\int Ed\rho\right) + C$. \\ Set $E_0 = 1 \wedge \inf_{\bm{x}\in \textcolor{black}{R}} E(\bm{x}) >0$ because \textcolor{black}{$\Sigma \cap U^c = \emptyset$, $X \times X \setminus U$ is a compact set} and $E$ is continuous. Then $\int_{U^c} Ed\rho \geq E_0 (1-\rho(U))$ which implies $-\ln\left(\int Ed\rho\right) \leq -\ln\left(\int_U Ed\rho + E_0 (1-\rho(U))\right)$. For $\alpha \in [0,1]$ and $x > 0$ set $f:x \mapsto -\alpha \ln(x) + \ln(x + E_0(1-\alpha))$. We have $f(x) \geq f(\alpha E_0) = -\alpha \ln(\alpha) + (1-\alpha) \ln(E_0) \geq \ln(E_0)$. This allows us to conclude that
    \begin{equation}
        H(\ga) \geq -\frac{d}{2}\ln\left(\int E(\bm{x}) d\rho\right) - \frac{d}{2}\ln\left(\frac{4\pi e}{\kappa d}\right)  + \frac{d}{2}\ln(E_0) + C(N, \Vert \Ph_i \Vert_{op},X).
    \end{equation}
\end{proof}

\subsection{Rates of convergence of $OT_\ep$}
The quadratic case showed that the duality gap function $E = c - (\ph\oplus\ps)$ has a global quadratic detachment, where $\ph,\ps$ are any Kantorovich potentials. In fact the duality gap still has a quadratic detachment, though local, in the more general case of infinitesimally twisted costs, which includes situations where the optimal plans are not necessarily given by a map. This is the point of \cite[Lemma 4.2]{carlier2022convergence}, which relies on a generalized version of Minty's trick also stating that in charts the support of $\ga_0$ is the graph of a Lipschitz function. They also provide a quadratic lower bound on the duality gap in these local charts. Our goal is to then apply Proposition \ref{prop:lower-bound-ent-local-detachment} in order to control the entropy of the optimal regularized transport plans as $\ep \to0$. We will then be able to conclude by finding a rate of order $\ep$ for $(c,\ga_\ep)-(c,\ga_0)$ in a similar fashion to the quadratic case.
\\Throughout this section we will work \textcolor{black}{on $X$, a compact subset  of $\mathbb{R}^d$.} We denote by $\Omega$ an open set containing $X$. We will assume that the marginals $\mu_0,\mu_1$ are supported in $X$ and have finite entropy. We start by recalling the results found in \cite{carlier2022convergence}. 

\begin{definition}[Infinitesimal twist]\label{def:twist}
    Given $c \in \mathcal{C}^2(\Omega^2)$ we say that $c$ is infinitesimally twisted if $\nabla^2_{xy}c(x,y) = (\partial^2_{x_iy_j}c(x,y))_{i,j} \in M_d(\mathbb{R})$ is invertible for every $(x,y) \in \Omega^2$.
\end{definition}
We now recall the quadratic detachment Lemma \cite[Lemma 4.2]{carlier2022convergence}. \textcolor{black}{Its proof} closely \textcolor{black}{follows} an earlier proof found in \cite{mccann2012rectifiability}, however \textcolor{black}{the authors} are interested in points that do not belong to the support of the optimal transport plan. 
\begin{lemma}\label{lem:mtw-minty-trick}
    Let $c \in \mathcal{C}^2(\Omega^2)$ be an infinitesimally twisted cost, and $(\ph,\ps) \in \mathcal{C}(X)^2$ be a pair of c-conjugate functions. We denote by $E = c - \ph \oplus \ps \geq 0$ the duality gap function defined on $X\times X$, by $\Sigma = \{ E= 0\}$ the contact set and for every $r >0$,
    \begin{equation}
        \tau(r) = \sup_{\substack{\bm{x},\bm{x'} \in X \times X\\ \Vert \bm{x'}-\bm{x}\Vert \leq r} } \Vert \nabla^2_{xy}c(\bm{x'})^{-1} \nabla^2_{xy}c(\bm{x}) - Id \Vert \in  [0, \infty).
    \end{equation}
    If $\bm{\bar{x}} \in X \times X$ and $\bm{x},\bm{x'} \in B_r(\bm{\bar{x}})\cap (X\times X)$, then
    \begin{equation}
        E(\bm{x})+E(\bm{x'}) \geq \Vert \Delta v \Vert^2 - \Vert\Delta u\Vert^2 - \tau(r) (\Vert \Delta v \Vert^2+\Vert \Delta u \Vert^2)
    \end{equation}
    where $\Delta u = u(\bm{x'}) - u(\bm{x}), \Delta v = v(\bm{x'}) - v(\bm{x})$, and
    \begin{equation}
        u(\bm{x}) = \frac{1}{2}(x-\nabla^2_{xy}c(\bm{\bar{x}}) y), \quad v(\bm{x}) = \frac{1}{2}(x+\nabla^2_{xy}c(\bm{\bar{x}}) y)
    \end{equation}
    for every $\bm{x} =(x,y)$.
\end{lemma}
\begin{remark}\label{rem:mtw-lipschitz-graph}
    The following remark was made in \cite{carlier2022convergence}. What directly follows from \textcolor{black}{the above lemma} is that locally, up to a change of  variable, the support of $\ga_0$ lies within the graph of a Lipschitz function, with Lipschitz constant arbitrarily close to $1$. Indeed for $(u,v),(u',v') \in \Ph(B(\bm{\bar{x}},r))\cap \Ph(\text{supp}\ga_0)$ we have
    \begin{equation}
        0= E(u,v) + E(u',v') \geq (1-\tau(r))\Vert \Delta v \Vert^2 - (1+\tau(r)) \Vert \Delta u \Vert^2.
    \end{equation}
    Thus by rearranging the terms in the inequality we have
    \begin{equation}
       \sqrt{\frac{1+\tau(r)}{1-\tau(r)}}  \Vert \Delta u \Vert \geq \Vert \Delta v \Vert.
    \end{equation}
\end{remark}
This Lemma essentially says that locally around the set $\Sigma = \{E = 0\}$ the function $E$ grows at least quadratically along the direction $\text{Im}(Id+\nabla^2_{xy}c(\bm{\bar{x}}))$. In particular when the functions $(\ph,\ps)$ are Kantorovich potentials of the optimal transport problem, this gives a quadratic lower bound of the duality gap close to $\Sigma$. The following Lemma restates Lemma \ref{lem:mtw-minty-trick} in the language of the quadratic detachment of the duality gap function. 
\begin{lemma}\label{lem:quad-detach-energy-gap}
    \textcolor{black}{Using the notations of Lemma \ref{lem:mtw-minty-trick}
    let} $(\ph,\ps) \in \mathcal{C}(X)^2$ be a pair of c-conjugate functions. Then $E=c - (\ph\oplus\ps)$ has a local quadratic detachment of parameters $(B(\bm{x}_i,r))_i,(\Ph_i)_i,\ka)$  (where $r$ is such that $\tau(r) \leq \frac{1}{2}$, and \textcolor{black}{$\ka =\frac{1}{2^{3-1/d}}\inf_{\bm{x}\in X^2} \vert \det(\nabla^2_{xy}c(\bm{x}))\vert^{1/d})$}. The \textcolor{black}{family} $(B(\bm{x}_i,r))_i$ is a finite covering of $\Sigma$, and $\Ph_i(\bm{x}) = \alpha_i(u_i(\bm{x})-u_i(\bm{x}_i),v_i(\bm{x})-v_i(\bm{x}_i))$ with $\alpha_i$ chosen such that \textcolor{black}{$\det(D\Ph_i) = 1$} and 
    \begin{equation}
        u_i(\bm{x}) = \frac{1}{2}(x - \nabla^2_{xy}c(\bm{x}_i) y), \quad v_i(\bm{x}) = \frac{1}{2}(x + \nabla^2_{xy}c(\bm{x}_i) y).
    \end{equation}
\end{lemma}

\begin{proof}
    We use the notations of Lemma \ref{lem:mtw-minty-trick}. Let $r > 0$ be such that $\tau(r) \leq 1/2$ which is possible by continuity of $\nabla_{xy}^2c$ over the compact set $X^2$. We have $\tau(r) \to 0$ as $r \to 0$. Note that $E$ is continuous because $\ph,\ps$ are $c$-conjugate functions and $c$ is continuous. Thus $\Sigma$ is closed, hence compact. We can choose $(\bm{x}_i)_{i=1}^N \in \Sigma^N$ such that $(B(\bm{x}_i,r))_i$ is a finite covering of $\Sigma$. For \textcolor{black}{$i \in \{1,\ldots,N\}$} and for every $\bm{x},\bm{x'} \in B(\bm{x}_i,r)$ such that $u_i(\bm{x})=u_i(\bm{x'})$, we have by Lemma \ref{lem:mtw-minty-trick}
    \begin{equation}
        E(\bm{x}) + E(\bm{x'}) \geq (1-\tau(r))\Vert \Delta v_i  \Vert^2 \geq \frac{1}{2\alpha_i^2}\Vert \Delta (\alpha_i v_i) \Vert^2
    \end{equation}
    where $v_i(\bm{x})=1/2(x+\nabla^2_{xy}c(\bm{x}_i)y)$. Note that $\Ph_i$ is an affine volume preserving map thus we have $1 = \frac{\alpha_i^{2d}}{2^d} \vert \det(\nabla_{xy}^2c(\bm{x}_i))\vert$. The determinant and the cross derivative are continuous which implies $I = \inf_{\bm{x}\in X^2} \vert \det(\nabla^2_{xy}c(\bm{x})) \vert> 0$ by the infinitesimal twist condition on $c$. We conclude that
    \begin{equation}
        E(\bm{x}) + E(\bm{x'}) \geq \frac{1}{2\alpha_i^2}\Vert \Delta (\alpha_i v_i) \Vert^2\geq  \textcolor{black}{\frac{I^{1/d}}{2^{3-1/d}}}\Vert \Delta (\alpha_i v_i) \Vert^2.
    \end{equation}
\end{proof}
We will use this Lemma in the specific case \textcolor{black}{of a} duality gap function where $\ph,\ps$ are Kantorovich potentials for the optimal transport problem. Note that there is no \textcolor{black}{dependence on} $\ep$ and thus the quadratic detachment framework can be used to derive the rates found in the quadratic cost case. 
 It is known (see \cite{nutz2022rate,carlier2022convergence}) that for infinitesimally twisted cost and compactly supported marginals of finite entropy, the regularized problem \textcolor{black}{\eqref{eq:EOT}} satisfies the following inequality for some real $C$
\begin{equation}\label{eq:gen-rate-eot}
    OT_\ep - OT_0 \leq - \frac{d}{2}\ep\ln(\ep) + C \ep.
\end{equation}
 We formally recall that convergence result in the following Lemma.

\begin{lemma}\label{lem:rate-otep-gen-case}\cite[Theorem 3.8, Lemma 3.13]{nutz2022rate}
     Assume that $\mu_i$ are compactly supported and $c\in \mathcal{C}^2(\Omega^2)$ is infinitesimally twisted. Then there is $C > 0$ such that, as $\ep \to 0$
    \begin{equation}
        OT_\ep - OT_0 \leq - \frac{d}{2}\ep\ln(\ep) + C \ep.
    \end{equation}
\end{lemma}

\begin{proof}
    Let $X$ be compact such that $\text{supp}(\mu_i) \subset X$. Then by continuity, $\nabla^2c$ is bounded. Moreover since $\mu_i$ are compactly supported \cite[Lemma 3.13]{nutz2022rate} ensures that the optimal transport plan $\ga_0$ satisfies the right quantization property for \cite[Theorem 3.8]{nutz2022rate} to apply and grant a constant $C > 0$ such that
    \begin{equation}
        OT_\ep - OT_0 \leq - \frac{d}{2}\ep\ln(\ep) + C \ep.
    \end{equation}
\end{proof}
Note that $OT_\ep - OT_0 = \int E d\ga_\ep + \ep H(\ga_\ep)$. Thus in combination with Lemma \ref{lem:rate-otep-gen-case} and the lower bound on the entropy \textcolor{black}{(Propostition \ref{prop:lower-bound-ent-local-detachment})} we have the following result.
\begin{theorem}\label{thm:general-rates}
    Let $c \in \mathcal{C}^2(\Omega^2)$ be infinitesimally twisted. Let $\mu_i \in \mathcal{P}_{ac}(\Omega)$ be two compactly supported measures satsifying $H(\mu_i) < \infty$. Then 
    \begin{equation}
    (c,\ga_\ep) - (c,\ga_0) = \Th(\ep), \quad H(\ga_\ep) = - \frac{d}{2} \ln(\ep) + O(1) \quad \emph{and} \quad OT_\ep = OT_0 - \frac{d\ep}{2} \ln(\ep) + O(\ep).
    \end{equation}
\end{theorem}
\begin{proof}
    We denote by $X$ a compact subset of $\Omega$ such that $\mu_i(X) = 1$. It is known that the Kantorvich potentials are $c$-conjugate functions and thus are continuous on $\Omega$. In particular the duality gap function $E(x,y) = c(x,y) - \ph(x) - \ps(y)$ is continuous on $\Omega\times \Omega$. By Lemma \ref{lem:quad-detach-energy-gap} we thus know that $E$ has a local quadratic detachment. Thus Proposition \ref{prop:lower-bound-ent-local-detachment} implies that there is a constant $C$ such that
    \begin{equation}\label{eq:lower-bound-entropy-gen-proof}
        H(\ga_\ep) \geq -\frac{d}{2}\ln\left(\int E d \ga_\ep\right) + C.
    \end{equation}
    Now combining this \textcolor{black}{inequality} with the upper bound on the rate of convergence \eqref{eq:gen-rate-eot} we have
    \begin{equation}\label{eq:core-inequation-gen-proof}
        - \frac{d}{2}\ep\ln(\ep) + C \ep \geq  \int E d\ga_\ep + \ep H(\ga_\ep) \geq \int E d\ga_\ep - \frac{d}{2} \ep\ln\left(\int E d\ga_\ep\right)  + C\ep.
    \end{equation}
    Dividing both sides by $\ep$ and adding $\ln(\ep)$ we get
    \begin{equation}
        \frac{\int E d\ga_\ep}{\ep} - \frac{d}{2}\ln\left(\frac{\int E d\ga_\ep}{\ep}\right) \leq C.
    \end{equation}
    And we conclude as in Theorem \ref{thm:rate-cost-quadratic}.
\end{proof}

\subsection{Lower bound on $W_2(\ga_\ep,\ga_0)$}

We now focus on the lower bound of the Wasserstein distance between $\ga_\ep$ and $\ga_0$. Unlike the quadratic case we \textcolor{black}{will not} be able to directly use a quadratic detachment for the Wasserstein distance. For $\ep > 0$, by construction $\ga_\ep$ is absolutely continuous with respect to the Lebesgue measure, thus there is an optimal transport map for the quadratic cost from $\ga_\ep$ to $\ga_0$. In particular this map is the gradient of a convex function $f$. We are now able to write the Wasserstein distance as $W^2_2(\ga_\ep,\ga_0) = \int \Vert \bm{x}- \nabla f(\bm{x})\Vert^2 d\ga_\ep(\bm{x})$. Inspired by the quadratic case and the last section we could say that $\bm{x}\mapsto \Vert \bm{x}- \nabla f(\bm{x})\Vert^2$ has a local quadratic detachment, but this function is not necessarily continuous which prevents us from applying the results on quadratic detachment. However the spirit of the proofs remains true and the results of this section are essentially an adaptation of the results presented before. First the Wasserstein distance between two measures satisfies a property close to a quadratic detachment whenever one of the measure is supported on a Lipschitz graph.

\begin{lemma}\label{lem:wasserstein-orthogonality}
Let $E$ be a subspace of $\mathbb{R}^n$, $T \ : E \rightarrow E^\perp$ and
 $\mu =(Id\times T)_\# \mu_0$ with $\mu_0$ a probability on $E$ with finite moment of order $2$. Suppose that T is L-Lipschitz. Then for any $\nu \in \mathcal{P}_{2,ac}(\mathbb{R}^d)$ 
\begin{equation}
W_2^2(\mu,\nu) \geq (1-L)W_2^2(\mu_0,\nu_0) + \frac{1}{L+1}\int \textcolor{black}{\emph{Var}}(\nu_x) d\nu_0(x)
\end{equation}
where $\nu= \nu_0 \otimes \nu_x$ is the disintegration of $\nu$ with regard to the orthogonal projection on $E$.
In particular 
\begin{equation}
(1+L)^2 W_2^2(\mu,\nu) \geq \int  \textcolor{black}{\emph{Var}}(\nu_x)d\nu_0(x).
\end{equation}
\end{lemma}

\begin{proof}
\textcolor{black}{Let $X,X'$ be random variables valued in $E$ and $Y,Y'$ be random variables valued in $E^\perp$.} Let $\pi$ be the optimal coupling between $\mu,\nu$.
Let $(X',Y') \sim \mu$ and $(X,Y) \sim \nu$ be such that $(X',Y',X,Y) \sim \pi$. Thus
\begin{equation}
W^2_2(\mu,\nu) = \mathbb{E}_\pi\left[\Vert X-X'\Vert^2 + \Vert Y - T(X')\Vert^2\right].
\end{equation}
Now using that $T$ is $L$-Lipschitz:
\begin{equation}
    L\Vert X-X'\Vert + \Vert Y-T(X')\Vert \geq \Vert Y - T(X) \Vert.
\end{equation}
and taking the square
\begin{equation}
    L^2\Vert X-X'\Vert^2 + \Vert Y-T(X')\Vert^2 + 2L\Vert X-X'\Vert\Vert Y-T(X')\Vert\geq \Vert Y - T(X) \Vert^2.
\end{equation}
Thus we have
\begin{equation}\label{eq:wass-mino-proof-lemma}
\begin{split}
    W^2_2(\nu,\mu) 
    &\geq \mathbb{E}_\pi\left[\Vert X-X'\Vert^2 + \Vert Y - T(X')\Vert^2\right]\\
    &\geq (1-L^2) \mathbb{E}_\pi\left[\Vert X-X'\Vert^2 \right] + \mathbb{E}_\pi\left[\Vert Y-T(X)\Vert^{2}\right] \\
    &- 2L\mathbb{E}_\pi\left[\Vert X-X'\Vert\Vert Y-T(X')\Vert\right]
\end{split}
\end{equation}
but the last term satisfies by Young inequality $ab \leq \frac{1}{2} (a^2 + b^2) $
\begin{equation}
    \mathbb{E}_\pi\left[\Vert X-X'\Vert\Vert Y-T(X')\Vert\right] \leq \frac{1}{2}(\mathbb{E}_\pi\left[\Vert X-X'\Vert^2\right] + \mathbb{E}_\pi\left[\Vert Y - T(X')\Vert^2\right]) \leq \frac{W_2^2(\mu,\nu)}{2}.
\end{equation}
Inequality \eqref{eq:wass-mino-proof-lemma} becomes
\begin{equation}
    W^2_2(\mu,\nu)(1+L) \geq  (1-L^2) \mathbb{E}_\pi\left[\Vert X-X'\Vert^2 \right] + \mathbb{E}_\pi\left[\Vert Y-T(X)\Vert^{2}\right] 
\end{equation}
and finally we have
\begin{equation}
\begin{split}
W^2_2(\mu,\nu) &\geq  (1-L) \mathbb{E}_\pi\left[\Vert X-X'\Vert^2 \right] +\frac{1}{1+L} \mathbb{E}_\pi\left[\Vert Y-T(X)\Vert^{2}\right] \\
 &\geq  (1 - L) W^2_2(\mu_0,\nu_0)+\frac{1}{1+L} \int\mathbb{E}_{\nu_x}\left[\Vert Y - \mathbb{E}\left[Y\mid X= x \right]\Vert^2\right] \textcolor{black}{d\nu_0(x)}\\
 & \geq (1 - L) W^2_2(\mu_0,\nu_0)+\frac{1}{1+L} \int \textcolor{black}{\text{Var}}(\nu_x)\textcolor{black}{d\nu_0(x)}
\end{split}
\end{equation}
\textcolor{black}{where} the second inequality is true since the mean is the \textcolor{black}{$L^2$-orthogonal} projection of the random variable on the space of constants. Now since $W^2_2(\mu,\nu) \geq W^2_2(\mu_0,\nu_0)$ we have
\begin{equation}
    (1+L) W^2_2(\mu,\nu)\geq W^2_2(\mu,\nu) + LW^2_2(\mu_0,\nu_0) \geq W^2_2(\mu_0,\nu_0)+\frac{1}{1+L} \int \textcolor{black}{\text{Var}}(\nu_x)\textcolor{black}{d\nu_0(x)}
\end{equation}
\textcolor{black}{which} finally grants $(1+L)^2 W^2_2(\mu,\nu) \geq \int \textcolor{black}{\text{Var}}(\nu_x)\textcolor{black}{d\nu_0(x)}$.
\end{proof}
The proof can be slightly modified to show that the function $\bm{x} \mapsto d(\bm{x},\Gamma)^2$ where $\Gamma$ is the graph of a Lipschitz function has a quadratic detachment. As stated before, to use the local quadratic detachment apparatus one could show that $\bm{x}\mapsto \Vert \bm{x}- \nabla f(\bm{x})\Vert^2$ has a local quadratic detachment where $\nabla f$ is the optimal transport from $\ga_\ep$ to $\ga_0$. And thanks to remark \ref{rem:mtw-lipschitz-graph}, $\ga_0$ is supported on a Lipschitz graph on charts. In particular if on an open set $U$, $\ga_0$ is the graph of a Lipschitz function, then $\bm{x}\mapsto \Vert \bm{x}- \nabla f(\bm{x})\Vert^2$ satisfies a quadratic detachment on $(\nabla f)^{-1}(U)$. However since we have no information on the regularity of the transport map we cannot conclude that $(\nabla f)^{-1}(U)$ is open and observe a local quadratic detachment. Thus we have to adapt the proof of Proposition \ref{prop:lower-bound-ent-local-detachment} in order to manage this issue and derive a lower bound for the entropy of $\ga_\ep$ using $W_2^2(\ga_\ep,\ga_0)$. 
\begin{proposition}\label{prop:gen-lower-bound-wass-gap}
    Let $\mu_i \in \mathcal{P}(\Omega)$ be two compactly supported measures of finite entropy \textcolor{black}{and $c \in \mathcal{C}^2(\Omega^2)$ be infinitesimally twisted.} Then there exists $C > 0$ such that as $\ep \to 0$
    \begin{equation}
        H(\ga_\ep) \geq -\frac{d}{2}\ln\left(W^2_2(\ga_\ep,\ga_0)\right) + C.
    \end{equation}
\end{proposition}
\begin{proof}
    Throughout we use the notations of Lemma \ref{lem:mtw-minty-trick}. Since $\ga_\ep$ has a density, there is a transport map $T$ for the optimal transport problem with quadratic cost from $\ga_\ep$ to $\ga_0$. Let $r > 0$ be such that $\tau(r) \leq \frac{1}{2}$. Let $(U_i = B(\bm{x_i},r)_i$ be an open covering of the support of $\ga_0$ with $\bm{x_i}\in \text{supp}\ga_0$. Let $(\zeta_i)_i$ be a partition of unity subordinate to the covering $(U_i)_i$. For $i$ set $\ga^i_\ep = \frac{1}{p_i}\zeta_i\circ T \ga_\ep$ and $\ga^i_0 =\frac{1}{p_i} \zeta_i \ga_0$, where $p_i$ is the normalization constant which is independent of $\ep$. Note that $T$ transports $\ga^i_\ep$ to $\ga^i_0$. Thus we have $W^2_2(\ga_\ep, \ga_0) \geq \sum_i p_iW_2^2(\ga^i_\ep,\ga^i_0)$. 
     We now introduce $\Ph_i$:
     $\Ph_i(x,y) = \alpha_i(u_i(x,y)-u_i(\bm{x_i}),v_i(x,y)-v_i(\bm{x_i}))$ where $\alpha_i$ is such that $\Ph_i$ is volume preserving.
      Let $\tau^i_\ep,\tau^i_0$ be the pushforwards of $\ga^i_\ep$ and $\ga^i_0$ with respect to this map. We disintegrate $\tau^i_\ep$ with respect to the projection on the variable $u$ and denote it $\mu^i_\ep \otimes \tau_\ep^{i,u}$. Note that $\Ph_i$ is not an isometry, however it is an affine map thus $W^2_2(\ga^i_\ep,\ga^i_0) \geq C_i W^2_2(\tau^i_\ep,\tau^i_0)$, where $C_i = 1/ \Vert \Ph_i \Vert_{op}$. Note that $C_i$ only depends on $\nabla^2_{xy}c(\bm{x}_i)$, thus it is independent of $\ep$. It was pointed out in remark \ref{rem:mtw-lipschitz-graph} that for every $i$, $\tau^i_0$ is supported on  the graph  of a $\sqrt{3}$-Lipschitz function. Thus thanks to Lemma \ref{lem:wasserstein-orthogonality}
    \begin{equation}
    \begin{split}
        H(\ga^i_\ep) = H(\tau^i_\ep) &= H(\mu^i_\ep) + \int H(\tau^{i,u}_\ep) d\mu^i_\ep \\
        &\geq H(\mu^i_\ep) - \frac{d}{2}\ln\left(\int  \textcolor{black}{\text{Var}}(\tau^{i,u}_\ep)d\mu^i_\ep \right) + C\\
        &\geq H(\mu^i_\ep) - \frac{d}{2}\ln\left(W^2_2(\ga^i_\ep,\ga^i_0) \right) + \frac{d}{2}\ln\left(\frac{C_i}{\textcolor{black}{(\sqrt{3}+1)^2}}\right)+C\\
        &\geq - \frac{d}{2}\ln\left(W^2_2(\ga^i_\ep,\ga^i_0) \right) + C(X)
    \end{split}
    \end{equation}
    \textcolor{black}{where} \textcolor{black}{the first inequality comes from inequality \eqref{eq:lower-bound-entropy-tau} and} the last inequality holds because $\mu^i_\ep$ is supported on a compact set of diameter comparable to that of $X$, as seen in the proof of \textcolor{black}{Proposition} \ref{prop:lower-bound-ent-local-detachment}.
    We proceed as in the proof of Proposition \ref{prop:lower-bound-ent-local-detachment} . Thus by summing over $i$ we get
    \begin{equation}
    \begin{split}
        H(\ga_\ep) &\textcolor{black}{\geq} \sum_i p_i \ln(p_i) + \sum_i p_i H(\ga^i_\ep)\\
        &\geq - \frac{d}{2}\sum_i p_i\ln\left(W^2_2(\ga^i_\ep,\ga^i_0) \right) + C\\
        &\geq - \frac{d}{2}\ln\left(W^2_2(\ga_\ep,\ga_0)\right) + C
    \end{split}
    \end{equation}
    \textcolor{black}{where} the last inequality holds by concavity of the logarithm.
\end{proof}
It remains to combine the last result with the rate of convergence of \eqref{eq:EOT} in order to retrieve the result on the Wasserstein distance between $\ga_\ep$ and $\ga_0$.

\begin{theorem}\label{thm:gen-mino-wass-any-coupling}
    Let $c \in \mathcal{C}(\Omega^2)$ be infinitesimally twisted. Let $\mu_i \in \mathcal{P}(\Omega)$ be two compactly supported measures of finite entropy. Then there exists $c > 0$ such that as $\ep \to 0$
    \begin{equation}
        W^2_2(\ga_\ep,\ga_0) \geq c \ep.
    \end{equation}
\end{theorem}

\begin{proof}

    Lemma \ref{lem:rate-otep-gen-case} ensures that
    \begin{equation}
        \ep H(\ga_\ep) \leq \int E d\ga_\ep + \ep H(\ga_\ep) \leq -\frac{d}{2}\ep \ln(\ep) + C\ep
    \end{equation}
    \textcolor{black}{and} Proposition \ref{prop:gen-lower-bound-wass-gap} grants the following lower bound for the entropy
    \begin{equation}
        H(\ga_\ep) \geq -\frac{d}{2} \ln(W^2_2(\ga_\ep,\ga_0)) + C.
    \end{equation}
    Combining the two \textcolor{black}{inequalities} together we have
    \begin{equation}
        -\frac{d}{2} \ln(W^2_2(\ga_\ep,\ga_0)) \leq -\frac{d}{2} \ln(\ep) + C.
    \end{equation}
    Taking the exponential grants the result.
\end{proof}
 Remark that in the general case, we obtain only the domination $\ep = O(W_2^2(\ga_\ep,\ga_0))$ whereas in the quadratic case, even if it is under strong assumptions, we could obtain $W_2^2(\ga_\ep,\ga_0) = \Th(\ep)$ (Theorem~\ref{thm:mino-wass-gen}). The reason is not that the solution are not smooth enough. Indeed, there exist some precise assumptions, such as the Ma-Trudinger-Wang conditions (see \cite{ma2005regularity}), that guarantee regularity of the solutions of the dual problem. The true difficulty is to replace the formula $W_2^2(\ga_\ep,\ga_0) \leq L (E,\ga_\ep)$ of Lemma \ref{lem:minty-wasserstein-upper-bound} that \textcolor{black}{transforms} this regularity (the \textcolor{black}{Lipschitz} constant L of the transport map) into \textcolor{black}{a bound on} $W_2^2(\ga_\ep,\ga_0)$. 
\printbibliography
\end{document}